\documentclass[12pt]{amsart}
\usepackage[utf8]{inputenc}
\usepackage{amsbsy}			% para símbolos matemáticos em negrito
\usepackage{amscd}			% para diagramas
\usepackage{amsfonts}		% fontes AMS
\usepackage{amsmath, mathtools}		% facilidades matemáticas
\usepackage{amssymb}	% para os símbolos mais antigos
\usepackage{amstext}	% para fragmentos tipo texto em modo matemático
\usepackage{amsthm}		% para teoremas
\usepackage{graphicx}
\usepackage{xcolor}
\usepackage{color}
\usepackage{hyperref}
\usepackage[all,cmtip]{xy}
\usepackage[normalem]{ulem}

\newcommand{\altair}[1]{\textcolor{blue}{#1}}
\newcommand{\francesco}[1]{\textcolor{red}{#1}}
\newcommand\restr[2]{{% we make the whole thing an ordinary symbol
		\left.\kern-\nulldelimiterspace % automatically resize the bar with \right
		#1 % the function
		\vphantom{\big|} % pretend it's a little taller at normal size
		\right|_{#2} % this is the delimiter
	}}
\newtheorem{thm}{Theorem}[section]
\newtheorem{lem}[thm]{Lemma}
\newtheorem{prop}[thm]{Proposition}
\newtheorem{cor}[thm]{Corollary}

\newtheorem{thmx}{Theorem}
\newtheorem{claim}[thm]{Claim}

\theoremstyle{definition}
\newtheorem{defi}[thm]{Definition}
\newtheorem{ex}[thm]{Example}

\theoremstyle{remark}
\newtheorem{rem}[thm]{Remark}

\begin{document}

\title{Conjugacy and Centralizers in Groups of Piecewise Projective Homeomorphisms}
	
\author{Francesco Matucci}
\thanks{The first author is a member of the Gruppo Nazionale per le Strutture Algebriche, Geometriche e le loro Applicazioni (GNSAGA) of the Istituto Nazionale di Alta Matematica (INdAM) and gratefully acknowledges the support of the 
		Funda\c{c}\~ao de Amparo \`a Pesquisa do Estado de S\~ao Paulo 
		(FAPESP Jovens Pesquisadores em Centros Emergentes grant 2016/12196-5),
		of the Conselho Nacional de Desenvolvimento Cient\'ifico e Tecnol\'ogico (CNPq 
		Bolsa de Produtividade em Pesquisa PQ-2 grant 306614/2016-2), and of
		the Funda\c{c}\~ao para a Ci\^encia e a Tecnologia  (CEMAT-Ci\^encias FCT projects UID/MULTI/04621/2019 and UIDB/04621/2020) and of the Universit\`a degli Studi di Milano - Bicocca
(FA project ATE-2016-0045 ``Strutture Algebriche'').}
\address{Università degli Studi di Milano-Bicocca, Italy}
\email{\href{mailto:francesco.matucci@unimib.it}{francesco.matucci@unimib.it}}

\author{Altair Santos de Oliveira-Tosti}
\thanks{This work is part of the second author's PhD thesis at the University of Campinas. 
		The second author gratefully acknowledges
		support from CNPq (grant 140876/2017-0) and Coordena\c{c}\~ao de Aperfei\c{c}oamento de Pessoal de N\'ivel Superior (CAPES)}
\address{Northern Paraná State University, Brazil}
\email{\href{mailto:altairsot@tutanota.com}{altairsot@tutanota.com}}

\begin{abstract}
	Monod introduced in \cite{Mon2013} 
	a family of Thompson-like groups which provides natural counterexamples to the  von  Neumann-Day conjecture. We construct a characterization of conjugacy and an invariant and use them to compute centralizers in one group of this family.
\end{abstract}

\maketitle

\markboth{Conjugacy and Centralizers in Groups of Piecewise Projective Homeomorphisms}{Francesco Matucci
and Altair Santos de Oliveira-Tosti}

\section{Introduction}

The von Neumann conjecture
states that a group is non-amenable if and only if it contains non-abelian free subgroups. It was formulated in 1957 by Mahlon Marsh Day and disproved in in 1980 by Alexander Ol'shanskii in \cite{Olshanskii1980} through a non-amenable Tarski monster group without any non-abelian free subgroup.
The historically first potential counterexample to such conjecture is Thompson's group $F$ of piecewise-linear homeomorphisms of the real line. The group $F$ does not contain any non-abelian free subgroup, but is still not known to be amenable.

Nicolas Monod introduced 
in \cite{Mon2013}
a class of groups
$H(A)$ depending on a subring $A$ of $\mathbb{R}$ providing
another family of counterexamples of the von Neumann-Day conjecture \cite{Mon2013}. Monod's groups are very natural and ``Thompson-like'' as they are described by
piecewise projective homeomorphisms of the real line. Later on Yash Lodha and Justin Moore \cite{LodhaMoore2016} found that \(H\left(\mathbb{Z}[1/\sqrt{2}]\right)\) contains a finitely presented subgroup, thus providing 
the first torsion-free finitely presented counterexample.

%Monod's groups and Thompson-like groups 

Thompson-like groups %are often topological full groups (like Monod's groups) and
have been extensively studied from the point of view of decision problems. Decision problems play an important role in group theory, giving a measure of the complexity of groups. A finitely presented group \(G\) 
    	has \textit{solvable conjugacy problem}, if there exists an algorithm which, given 
    	\(y,z\in G\), determines whether or not there is an element \(g\in G\) such 
    	that \(g^{-1}yg=z\). This problem has been studied for many classes of groups
    	and is generally unsolvable. 
    	The conjugacy problem has been studied for several Thompson-like groups 
    	\cite{BarDunRob2016,BelkMat2014,BurMatVent2016,GillShort2013,GubaSapir1997,Higman1974,KassMat2012, Mat2010,Robertson2019,Salazar2010}.
    	Monod's groups share commonalities used in approaches used to study 
    	the conjugacy problem, such as being a topological full group. In this paper we exploit such commonalities to understand conjugacy in Monod's group $H\coloneqq H(\mathbb{R})$ and find a 
    	criterion (Corollary \ref{thm:characterize-conjugacy} below)
    	to establish 
    	conjugacy within the group.
    	
    	Matthew Brin and Craig Squier construct in \cite{BrinSquier2001} 
    a conjugacy invariant in the infinitely generated 
    group $\mathrm{PL}_{+}(\mathbb{R})$ of all piecewise-linear homeomorphisms of the real line with finitely many breakpoints
    and use it to compute
    element centralizers by adapting techniques developed in \cite{Mather1974}.
    This invariant has been revisited later in \cite{GillShort2013,Mat2010} and we adapt it in Theorem \ref{matconj} below to produce our own version of this invariant and compute centralizers:
    \begin{thmx}\label{thm:intro-a}
    	Given \(z\in H\), then
    	\[C_{H}\left(z\right)\cong\left(\mathbb{Z},+\right)^{n}\times\left(\mathbb{R},+\right)^{m}\times H^{k},\]
    	for suitable \(k,m,n\in\mathbb{Z}_{\geq0}\).
    \end{thmx}
    
    Several of our results adapt to the general Monod groups $H(A)$ for a subring
    $A$ of $\mathbb{R}$, but there are some for which the proofs given for $H$ do not immediately
    apply to the groups $H(A)$. More precisely, the results of Section \ref{sec:stair} can be easily rephrased and proved for $H(A)$, while those from Sections \ref{sec:mather} and \ref{sec:centralizer} may extend too, but our proofs
    do not apply to $H(A)$.
    
    The work is organized as follows: in Section \ref{sec:monod}, we 
    	define Monod groups and present some basic properties, some of which shared with 
    	Thompson's group \(F\). In Section \ref{sec:stair}, we discuss a characterization of 
    	conjugacy, which is an 
    	adaptation of the \textit{Stair Algorithm}, developed by Kassabov and the 
    	first author in \cite{KassMat2012}.
    	In Section \ref{sec:mather}, we define a conjugacy invariant (the \textit{Mather invariant}) for a class of elements by adapting 
    	techniques developed in \cite{Mather1974}, and we show the 
    	relation between the Stair Algorithm and the Mather invariant.
    	In Section \ref{sec:centralizer}, we compute the centralizer subgroups
    	of elements from \(H\) as applications of the preceding tools. 
    	
	\noindent \textbf{Acknowledgments:} 
		The authors would like to thank	James Belk, Martin Kassabov and Slobodan 
		Tanushevski for helpful conversations about the present work. The authors
		would also like to thank an anonymous referee for helpful suggestions
		that improved the readability of the paper.

	\section{Monod's Groups}\label{sec:monod}
	
	In this section, we will discuss groups of piecewise
	projective orientation-preserving homeomorphism of
	$\mathbb{R}\mathrm{P}^{1}$ which stabilize infinity and 
	discuss some of their properties. These groups are called 
	\textit{Monod's groups} and they were introduced by 
	Nicolas Monod in \cite{Mon2013}.
	
%%%%%%%%%%%%%%%%%%%%%%
%%%%%%%%%%%%%%%%%%%%%%
\iffalse
%%%%%%%%%%%%%%%%%%%%%%
%%%%%%%%%%%%%%%%%%%%%%

	%We stress that most results in this section are not new.
	and provide another family 
		of counterexamples of the von Neumann-Day conjecture \cite{Mon2013} stating that all non-amenable groups contained free subgroups.
		Thompson's group \(F\) had been suggested as one as a potential counterexample to this 
		conjecture as it contains no free subgroups, but it is still undetermined whether or not $F$ is amenable. In 1980 Alexander Ol'shanskii constructed a non-amenable Tarski monster group without any free subgroup  \cite{Olshanskii1980}, disproving the von Neumann-Day conjecture. Since then, many 
		other counterexamples were constructed later, but question as to whether a
		Thompson-like counterexample existed. Monod's groups provide such counterexamples and it was later found out by Lodha and Moore \cite{LodhaMoore2016} that \(H\left(\mathbb{Z}[1/\sqrt{2}]\right)\) contains a finitely 
		presented subgroup which is itself a counterexample of 
		von Neumann-Day conjecture.

%%%%%%%%%%%%%%%%%%%%%%
%%%%%%%%%%%%%%%%%%%%%%
\fi
%%%%%%%%%%%%%%%%%%%%%%
%%%%%%%%%%%%%%%%%%%%%%

	We now introduce the notation that will be used in the paper. 
	%If \(A\)
	%\footnote{It is worth reminding the reader that when we use the word ``subring'' in this paper, we always mean that the multiplicative identity of \(A\) is the same of the field \(\mathbb{R}\).} 
	If \(A\) is a subring of
	\(\mathbb{R}\) with unit, the group of M\"obius transformations
	\(\mathrm{PSL}_{2}\left(A\right)\), under composition of functions, is the group of transformations of the real projective line
	%orientation-preserving homeomorphisms
	$\mathbb{R}\mathrm{P}^{1}=\mathbb{R}\cup \{\infty\}$
	of the form $f\colon t \mapsto \dfrac{at+b}{ct+d}$ for $a,b,c,d \in A$
	where the determinant of the associated matrix $M_f=
	\left(
	\begin{matrix}
	a & b \\
	c & d
	\end{matrix}
	\right)
	$ is equal to $1$. We say that $f$ is \textbf{hyperbolic} if $|\mathrm{tr}(M_f)|>2$.
	We consider the group $\mathrm{PPSL}_2(A)$ of piecewise projective homeomorphisms of $\mathbb{R}\mathrm{P}^{1}$ with multiplication given by composition of functions. We say that \(f\in \mathrm{PPSL}_2(A)\) 
	if there are finitely many 
	points \(t_{0},t_{1},\ldots,t_{n+1}\in\mathbb{R}\mathrm{P}^{1}\)
	so that on each interval \((-\infty,t_0]\), $[t_{i},t_{i+1}]$, \(i=0,1,\ldots,n-1\) and $[t_n,\infty)$ the map is a M\"obius transformation
	\[f\colon t \mapsto \dfrac{a_it+b_i}{c_it+d_i},
	\]
	\noindent where \(a_id_i-c_ib_i=1\), for suitable \(a_i,b_i,c_i,d_i \in A\). \emph{Monod's group} $H(A)$ is the subgroup of $\mathrm{PPSL}_2(A)$ where \(f(\infty)=\infty\) and the points $t_0, \ldots, t_n$ lie in the set \(\mathcal{P}_{A}\) 
	of fixed points of hyperbolic M\"obius transformations in
	\(\mathrm{PSL}_{2}\left(A\right)\). In the case $A=\mathbb{R}$ we simply write $H(\mathbb{R})=H$.
	We say that a point $t_{0}\in \mathcal{P}_{A}$ 
	is a \textbf{breakpoint} of $f\in\mathrm{PPSL}(A)$ if 
	there exists an $\varepsilon >0$ such that
	there do not exist $a,b,c,d \in A$, where
	$ad-cb=1$ and $f(t)= \dfrac{at+b}{ct+d}$
	on $(t_{0}-\varepsilon,t_{0}+\varepsilon)$.	

	One of the requirements to adapt the Stair Algorithm to this setting is to be able to simultaneously send a tuple of intervals to another such tuple, which means having a form of transitivity.
	%Since the idea of Stair Algorithm is to build the candidate conjugator between 
	%two elements, we need that all its interval endpoints lie in \(\mathcal{P}_{A}\). 
	%In addition, since in the case of \(A=\mathbb{R}\) we get 
	%\(\mathcal{P}_A=\mathbb{R}\mathrm{P}^{1}\), we can prove that it is possible to 
	%associate any two $k$-tuples of values from \(\mathbb{R}\) via some element from 
	%\(H\). 
	We need \(H\) to act order \(k\)-transitively on 
	\(\mathbb{R}\mathrm{P}^{1}\) and this is a property shared with Thompson's group 
	\(F\). The proof of the following result is analogous to the one for $\mathrm{PL}_+(\mathbb{R})$.
	
	\begin{lem}\label{monod-transitive}
	Let \(t_{1}<t_{2}<\ldots<t_{k}\) and \(s_{1}<s_{2}<\ldots<s_{k}\) be elements 
		from \(\mathbb{R}\mathrm{P}^{1}\setminus \{\infty\}\). Then there exists 
		\(f\in H\) such that \(f(t_{i})=s_{i}\), for all \(i=1,2,\ldots,k\).
	\end{lem}
\begin{proof}
	For all $i\in\left\{1,2,\ldots,k-1\right\}$, 
	let us consider the intervals $[t_{i},t_{i+1}]$ and
	$[s_{i},s_{i+1}]$. Since $\mathrm{PSL}_2(\mathbb{R})$ 
	is $2$-transitive on $\mathbb{R}\mathrm{P}^{1}$ 
	(see \cite[Theorem 5.2.1 $(ii)$]{JonSin1987})
	there exists an element $f_{i}\in \mathrm{PSL}_2(\mathbb{R})$ 
	such that
	\begin{displaymath}
	f_{i}(t_{i})=s_{i}\;\;\textrm{and}\;\;f_{i}(t_{i+1})=s_{i+1}.
	\end{displaymath}
	\noindent 
	Thus, it is enough to glue together these maps with two
	functions $f_{0},f_{k}\in\mathrm{PSL}_2\left(\mathbb{R}\right)$ 
	defined on $\left(-\infty,t_{1}\right]$ and $\left[t_{k},+\infty\right)$,
	respectively, as
	\begin{displaymath}
	f_{0}(t)=\dfrac{a_{0}t+b_{0}}{d_{0}}\;\;\textrm{and}\;\;f_{k+1}(t)=\dfrac{a_{k}t+b_{k}}{d_{k}},
	\end{displaymath}
	\noindent 
	where $a_{0}d_{0}=a_{k}d_{k}=1$ and
	$a_{0},b_{0},d_{0},a_{k},b_{k},d_{k}$ are chosen in such 
	way that $f_{0}(t_{1})=s_{1}$ and $f_{k}(t_{k})=s_{k}$.
	To finish, we construct the following element from $H$
	\[
	f(t)\coloneqq\begin{dcases}
	f_{0}(t),&\;\textrm{if}\;t\in\left(-\infty,t_{1}\right]\\
	f_{i}(t),&\;\textrm{if}\;t\in\left[t_{i},t_{i+1}\right]\\
	f_{k}(t),&\;\textrm{if}\;t\in\left[t_{k},+\infty\right)
	\end{dcases}
	\]
	for $i\in\left\{1,2,\ldots,k-1\right\}$, so that $f(t_{i})=s_{i}$, for all $i\in\left\{1,2,\ldots,k\right\}$.
\end{proof}

\begin{rem}
The proof that Lemma \ref{monod-transitive} is true for $H$ 
does not immediately carry over to $H(A)$, for a subring $A$ of $\mathbb{R}$, as we are not aware of a transitivity result for
fixed points of hyperbolic Mobi\"us transformations. In this paper we sometimes make use of Lemma \ref{monod-transitive} and, in these instances, our proofs do not immediately carry over to $H(A)$, although
it is not clear that they cannot be achieved through a different route. Several of the results
of this paper carry over to $H(A)$, while for others we cannot immediately say that they do.
\end{rem}
	If \(f\in H\left(A\right)\), there are finitely 
	many points \(t_{1},t_{2},\ldots,t_{n} \in \mathcal{P}_{A}\) such that on each 
	interval \(\left(-\infty,t_{1}\right]\), \(\left[t_{i},t_{i+1}\right]\) for $i=1,\ldots,n-1$, and \(\left[t_{n},+\infty\right)\) we have
	\(f\colon t \mapsto \left(a_it+b_i\right)/\left(c_it+d_i\right)\), where 
	\(a_id_i-c_ib_i=1\), for suitable \(a_i,b_i,c_i,d_i\in A\). Since $f(\pm \infty)=\pm \infty$, 
	we must have $c_1=c_n=0$ and so
	 \(f\colon t \mapsto (a_0t+b_0)/d_0\quad\textrm{and}\quad f\colon t \mapsto (a_nt+b_n)/d_n\)
	on \(\left(-\infty,t_{1}\right]\) and \(\left[t_{n},+\infty\right)\), 
	respectively, where \(a_{0}d_{0}=1=a_{n}d_{n}\), for
	\(a_{0}, a_{n}, b_{0},b_{n}\in A\). Then we can say that elements in 
	\(H\left(A\right)\) have \textbf{affine germs} at \(\pm\infty\). In other words, 
	when \(t\in\left(-\infty,t_{1}\right]\) we rewrite \(f\) in this interval as 
	\(f\left(t\right)=a_{0}^{2}t+a_{0}b_{0}\), for all 
	\(t\in \left(-\infty,t_{1}\right]\), since \(a_{0}d_{0}=1\). Similarly, we can 
	rewrite \(f\) as \(f\left(t\right)=a_{n}^{2}t+a_{n}b_{n}\), for all
	\(t\in \left[t_{n},+\infty\right)\), since \(a_{n}d_{n}=1\). 
	\begin{rem}\label{Units-Translations}\cite{BurLodRee2018}
		Notice that, for all elements in \(H\left(A\right)\), the germs at infinity 
		satisfy that the slopes \(a_{0}^{2}\) and \(a_{n}^{2}\) are units of the ring 
		\(A\). Thus, if the only units of \(A\)	are \(\pm1\), the first and last 
		parts of maps in \(H\left(A\right)\) are translations. For instance, if 
		\(A = \mathbb{Z}\), the only possibility is that \(a_{0}^{2}=a_{n}^{2}=1\).
	\end{rem}
	
	A property that is inherently used while studying the conjugacy problem in the works \cite{KassMat2012,Mat2010} which we will
	adapt to work for Monod's group \(H\) is that the Thompson-Stein groups $\mathrm{PL}_{A,G}(I)$,
	defined for a subring $A$ of $\mathbb{R}$ and a subgroup $G$ of the positive units of $A$, are  full groups.
	
	\begin{defi}\index{Full group}
	Let $G$ be a group of homeomorphisms of
	some topological space $X$.
	\begin{itemize}
		\item[(a)] A homeomorphism $h$
		of $X$ \textbf{locally agrees} 
		with $G$ if for every point $p\in X$,
		there exists a neighborhood $U$ of $p$
		and an element $g\in G$ such that 
		\[\restr{h}{U}=\restr{g}{U}.\]
		We denote the set of all homeomorphisms
		of $X$ which locally agree with $G$
		by $\left[G\right]$;
		\item[(b)] The group $G$ is a
		\textbf{full} if every
		homeomorphism of $X$ 
		that locally agrees	with $G$ 
		belongs to $G$. In other words, 
		$G$ is a full group if 
		$G=\left[G\right]$.
	\end{itemize} 
\end{defi} 
\begin{lem}
	Monod's group $H(A)$ is a full group for any subring $A$ of $\mathbb{R}$.
\end{lem}
\begin{proof}
Given a subring $A$ of $\mathbb{R}$ and $h \in [H(A)]$, compactness of $\mathbb{R}\mathrm{P}^{1}$ implies that $h$ has only finitely many breakpoints, as it locally agrees with maps from $H(A)$.
	Moreover, $h$ must have affine germs around $\pm \infty$, 
 since it coincides with some element from $H(A)$ and so $h\in H(A)$. Therefore $[H(A)]\subseteq H(A)$. and so $H(A)$ is a full group.
\end{proof}

We finally recall another property of Monod's group which is shared with Thompson's
group $F$ (see \cite{Mon2013}).
\begin{lem}
\label{thm:torsion-free}
Monod's group $H(A)$ is torsion-free for any subring $A$ of $\mathbb{R}$.
\end{lem}
	
	For more properties of Monod's groups, we encourage the interested reader to consult the references 
		\cite{BurLodRee2018, Mon2013}.

	\section{The Stair Algorithm}\label{sec:stair}
	
	In this section, we adapt the \textit{Stair Algorithm} developed in 
	\cite{BurMatVent2016,KassMat2012}. If there exists a conjugator between two elements, this algorithm 
	allows us to construct such conjugator from an ``initial germ''. %which we will describe Definition \ref{def:germs}. 
	The algorithm constructs the conjugator by looking at necessary conditions
	it should satisfy and building it piece by piece until we reach the so-called ``final box'' and ending the construction. We show that, if a conjugator exists, it has to coincide with the homeomorphism we construct. In the
    following, if \(y,z\in H\) and there is a \(g\in H\) such that \(g^{-1}yg=z\), we will write \(y^{g}=z\).

	%First, we need adapt some initial definitions and 
	%necessary conditions for conjugacy, as we will see in Lemmas 
	%\ref{monod-initialslope} and \ref{monod-germsconj}. The main idea is to reduce 
	%our search to some families of maps in \(H\). This reduction arises from the 
	%analysis of the set of fixed points of elements from \(H\).	Next, we have the 
	%Lemma \ref{monod-initialbox}, which determines uniquely the first piece of a 
	%possible conjugator, given its initial germ. Then we constructed a conjugator 
	%step-by-step from its initial germ. We will identify \(y\) and \(z\) inside a box 
	%close to the initial box by a suitable conjugator. Then we repeat this process in 
	%more boxes until we reach the final box. This algorithm ends after a finite 
	%number of steps.
	
	\subsection{Notations} Let us fix some notation. Given a $h \in H$, we define the \textbf{support} of $h$ to be $\mathrm{supp}(f)=\{t \in \mathbb{R} \mid f(t) \ne t\}$.
	
	\begin{defi}
		Let \(G\) be any subset of \(H\). We define \(G^{>}\) as the subset of \(G\)
		of all maps that lie above the diagonal, that is,
		\[G^{>}\coloneqq\left\{g\in G\mid  g\left(t\right)>t,\; t\in\mathbb{R} \right\}.\]
		Similarly, we define \(G^{<}\). A homeomorphism $g \in G^{>} \cup G^{<}$
		is called a \textbf{one-bump function}. 
		Moreover, for every \(-\infty\leq p<q \leq+\infty\), we define
		\(G\left(p,q\right)\) as the set of elements of \(G\) with support contained inside 
		\(\left(p,q\right)\), that is,
		\[G(p,q)\coloneqq\left\{g\in G \mid g\left(t\right)=t,\; t\notin(p,q) \right\}.\]
		We also define the subset
		\[G^{>}(p,q)\coloneqq\left\{g\in G \mid g\left(t\right)=t,\;\forall\; t\notin(p,q)\;\text{and}\;g\left(t\right)>t,\;\forall\;t\in(p,q)\right\}.\]
		Analogously, we define \(G^{<}\left(p,q\right)\). If $g \in G^{>}\left(p,q\right) \cup G^{<}\left(p,q\right)$, we say that \(g\) 
		is a \textbf{one-bump function on \(\left(p,q\right)\)}.
	\end{defi}
	\begin{rem}
		If \(G\) is a subgroup, then \(g\in G^{>}\) if, and only if, \(g^{-1}\in G^{<}\).
	\end{rem}
	\indent Since elements \(f\in H\) are defined for all real numbers, we will define
	suitable ``boxes'' for real numbers around \(\pm\infty\). In order to work 
	with numbers sufficiently close to \(\pm\infty\), we give the next 
	definition.
	\begin{defi}
		A property \(\mathcal{P}\) holds \textbf{for \(t\) negative sufficiently 
		large} (respectively, \textbf{for \(t\) positive sufficiently large}) to mean 
		that there exists a real number \(L<0\) such that \(\mathcal{P}\) holds for 
		every \(t\leq L\) (respectively, there is a positive real number \(R\) so 
		that \(\mathcal{P}\) holds for every \(t\geq R\)).
	\end{defi}
	\subsection{Necessary Conditions}
	In \cite{KassMat2012}, Kassabov and the first author worked with the initial and final slopes of 
	elements from \(\mathrm{PL}_{+}\left(\left[0,1\right]\right)\). If two 
	elements from \(\mathrm{PL}_{+}\left(\left[0,1\right]\right)\) are conjugate, they coincide on suitable	``boxes'' around $0$ and $1$. Let us 
	define similar concepts for elements from \(H\).
	
	Given \(y\in H\), let us denote the slope of \(y\) for \(t\) negative 
	sufficiently large as
	\[y'\left(-\infty\right)\coloneqq \lim\limits_{t\rightarrow -\infty}y'\left(t\right).\]
	Similary, we denote the slope of \(y\) for \(t\) positive sufficiently large as 
	\(y'\left(+\infty\right)\). However, if two elements from \(H\)	have the same 
	slopes for \(t\) negative sufficiently large, they do not necessarily coincide 
	around \(-\infty\). Thus, in order to ensure that two elements coincide for \(t\) 
	negative sufficiently large, we give the following definition.
	\begin{defi}\label{def:germs}
		We define the \textbf{germ of \(y \in H\) at \(-\infty\)} as the pair
		\[y_{-\infty} \coloneqq \left(y'\left(-\infty\right),y\left(L\right)-y'\left(-\infty\right)L\right),\]
		\noindent where \(L\) is the largest real number for which \(y\) is the 
		affine map with slope \(y'\left(-\infty\right)\) on the interval 
		\(\left(-\infty,L\right]\). If $y$ is affine on $\mathbb{R}$, then we $L$ can taken to be any real number.
		We call \(y_{-\infty}\) the \textbf{initial germ}.
		Analogously, we define the \textbf{final germ} \(y_{+\infty}\). 
	\end{defi}
	\indent We remark that, for an element \(y\in H\), the initial germ \(y_{-\infty}\) and the final germ \(y_{+\infty}\) are 
	elements of the \textbf{affine group} \(\mathrm{Aff}\left(\mathbb{R}\right)\), 
	which is defined as the semidirect product 
	\(\mathrm{Aff}\left(\mathbb{R}\right)\coloneqq\mathbb{R}_{>0}\ltimes\mathbb{R}\), 
	where \(\mathbb{R}_{>0}\) denotes the multiplicative group 
	\(\left(\mathbb{R}_{>0},\cdot\right)\) and \(\mathbb{R}\) denotes the additive 
	group \(\left(\mathbb{R},+\right)\). The operation of this group is
	\((a,b)(c,d)\coloneqq (ac,b+ad)\). The identity element is \(\left(1,0\right)\)
	and inverses are given by \(\left(a,b\right)^{-1}=\left(a^{-1},-a^{-1}b\right)\).
	 
	%and we sometimes denote them by
	%$y_{-\infty}^{\mathrm{Aff}\left(\mathbb{R}\right)_{-\infty}}$ and $y_{+\infty}^{\mathrm{Aff}\left(\mathbb{R}\right)_{+\infty}}$. {\large \francesco{Is it really necessary to use the notation 
	%$y_{-\infty}^{\mathrm{Aff}\left(\mathbb{R}\right)_{-\infty}}$ (like in a lemma below)? Can't we just keep calling it $y_{-\infty}$?}}
	
	The following observation on slopes is the first necessary condition we test for conjugacy. Its proof is a straightforward calculation.
	\begin{lem}\label{monod-initialslope}
		Let \(y,z\in H\) be such that \(y^{g}=z\), for some \(g\in H\). Then for \(t\) 
		negative (respectively, positive) sufficiently large we have 
		\(y'\left(-\infty\right)=z'\left(-\infty\right)\) 
		(respectively, \(\left(y'\left(+\infty\right)=z'\left(+\infty\right)\right)\)).
	\end{lem}
	
	The next necessary condition we observe if that if
	the conjugacy classes of the germs of 
	\(y,z\in H\) at \(-\infty\), or at \(+\infty\), are different, then \(y\) and 
	\(z\) cannot be conjugate.
	 %between two
%	elements from \(H\).
	
	\begin{lem}\label{monod-germsconj}
		For any \(y,z\in H\) such that \(y^{g}=z\) for some \(g\in H\), then the conjugacy classes
		$y_{-\infty}^{\mathrm{Aff}\left(\mathbb{R}\right)}$ and $z_{-\infty}^{\mathrm{Aff}\left(\mathbb{R}\right)}$ of $y_{-\infty}$ and $z_{-\infty}$ inside $\mathrm{Aff}\left(\mathbb{R}\right)$ coincide.
		Similarly, we have \(y_{+\infty}^{\mathrm{Aff}\left(\mathbb{R}\right)}= z_{+\infty}^{\mathrm{Aff}\left(\mathbb{R}\right)}\).
	\end{lem}
	\begin{proof}
		Assume that \(g_{-\infty}=\left(a^{2},ab\right)\) and \(y_{-\infty}=\left(a_{0}^{2},a_{0}b_{0}\right)\).
		Since $y^g=z$,
		it is straightforward to see that, for \(t\) negative sufficiently large, we have
		that 
		\begin{align*}
		z_{-\infty}=(y^g)_{-\infty}=y_{-\infty}^{g_{-\infty}}=(a^{-2},-a^{-1}b)\cdot(a_{0}^{2},a_{0}b_{0})\cdot
		(a^{2},ab)=\\ \left(a_{0}^{2},a_{0}b_{0}a^{-2}+\left(a_{0}^{2}-1\right)a^{-1}b\right)
		\end{align*}
		Thus	 \(y_{-\infty}^{\mathrm{Aff}\left(\mathbb{R}\right)}= z_{-\infty}^{\mathrm{Aff}\left(\mathbb{R}\right)}\).
		Similarly, we see
		\(y_{+\infty}^{\mathrm{Aff}\left(\mathbb{R}\right)}= z_{+\infty}^{\mathrm{Aff}\left(\mathbb{R}\right)}\).
	\end{proof}
	
 From now on, if \(y_{-\infty}\) and \(z_{-\infty}\)
	are conjugate in \(\mathrm{Aff}\left(\mathbb{R}\right)\), we will denote it by 
	\(y_{-\infty}\sim_{\mathrm{Aff}\left(\mathbb{R}\right)}z_{-\infty}\).

%	\subsection{Making Fixed Points Coincide}
%	{\huge \textcolor{red}{Should I put this subsection?}}
    
	\subsection{Initial and final boxes}
	In this subsection, we see that a possible conjugator between two given 
	elements is determined by its germs inside suitable boxes. 

	\begin{lem}[Initial and final boxes]\label{monod-initialbox}
		Let \(y,z\in H^{>}\left(-\infty,p\right)\) for some \(-\infty<p\leq+\infty\)
		and let \(g\in H\) be such that \(y^{g}=z\). Then there exists a constant \(L\in\mathbb{R}\) (depending on \(y\) and \(z\)) such that 
		\(g\) is affine on the initial box 	\(\left(-\infty,L\right]^{2}\). An 
		analogous result holds, for \(y,z\in H^{>}\left(p,+\infty\right)\) for some 
		\(-\infty\leq p<+\infty\) and a final box \(\left[R,+\infty\right)^{2}\).
	\end{lem}
	
\begin{proof}
	By Lemma \ref{monod-initialslope}, there exists an $L<\min\{0,p\}$ such that \(y'\left(t\right)=z'\left(t\right)\) for \(t\leq L\). Up to replacing $L$ by a suitable $L_1 < L$ we can assume that \(y'\left(t\right)=z'\left(t\right)\) for every 
	\(t\leq L\). 
	Assume that \(g_{-\infty}=\left(a^{2},ab\right)\) and \(y_{-\infty}=\left(a_{0}^{2},a_{0}b_{0}\right)\) then, 
	following the same calculations of Lemma \ref{monod-germsconj}, we have
	\[y\left(t\right)=a_{0}^{2}t+a_{0}b_{0}\;\text{and}\; z\left(t\right)=a_{0}^{2}t+a_{0}b_{0}a^{-2}+a^{-1}b(a_{0}^{2}-1)\]
	for all \(t\leq L\) and for suitable \(a,b\in\mathbb{R}\).
	
	We can rewrite our goal as follows: if we define
	\[\widetilde{L}\coloneqq \sup\left\{r \mid g\;\text{is affine on}\;\left(-\infty,r\right]\right\},\]
	\noindent then \(\widetilde{L}\geq\min\left\{L,g^{-1}\left(L\right)\right\}\). 
	Let us assume the opposite, that is, \(\widetilde{L}<\min\left\{L,g^{-1}\left(L\right)\right\}\)
	and 
	\[g\left(t\right)=\begin{cases}
	a^{2}t+ab,&\;\text{if}\;t\in\left(-\infty,\widetilde{L}\right], \\
	\dfrac{\bar{a}t+\bar{b}}{\bar{c}t+\bar{d}},&\;\text{if}\;t\in\left[\widetilde{L},L_2\right),
	\end{cases}\]
	\noindent for suitable $\bar{a},\bar{b},\bar{c},\bar{d} \in \mathbb{R}$ and $\widetilde{L} < L_2 \le L$ so that $g$ has a breakpoint at $\widetilde{L}$. Without loss of generality, we can assume that $L_2 = L$.
	
	Since \(\widetilde{L}<L<0\) and	\(z \in H^{>}\left(-\infty,p\right)\), we have $L<z(L)$ and so there exists
	a real number \(\sigma>1\) such that \(\sigma\widetilde{L}<\widetilde{L}<L\)
	and \(\widetilde{L}<z(\sigma\widetilde{L})<L\).
	%Then \(z\) is affine around	\(\sigma\widetilde{L}\) and
	%\begin{equation}\label{linbox-eq1}
	%gz(\sigma\widetilde{L})=g(z(\sigma\widetilde{L})).
	%\end{equation}
	On the other hand, \(\widetilde{L}<g^{-1}\left(L\right)\) and so 
	\(\sigma\widetilde{L}<g^{-1}\left(L\right)\). Thus we have 
	\(g(\sigma\widetilde{L})<L\), which means that \(y\) is affine around 
	\(g(\sigma\widetilde{L})\) and
	\begin{align}\label{linbox-eq1}
		y(g(\sigma\widetilde{L}))&=
		y(a^{2}\sigma\widetilde{L}+ab) =
		a^{2}(a_{0}^{2}\sigma\widetilde{L})+a_{0}^{2}ab+a_{0}b_{0}\nonumber\\
		&=a^{2}(a_{0}^{2}\sigma\widetilde{L}+a^{-2}a_{0}b_{0}+a^{-1}ba_{0}^{2}-a^{-1}b+a^{-1}b)\\
		&=a^{2}z(\sigma\widetilde{L})+ab.\nonumber
	\end{align}

	Since \(gz\left(t\right)=yg\left(t\right)\) for every real number \(t\), equation 
	\(\left(\ref{linbox-eq1}\right)\) returns
	\begin{align}\label{linbox-eq2}
	g(z(\sigma\widetilde{L}))=a^{2}(z(\sigma\widetilde{L}))+ab,
	\end{align}
	for any real number \(\sigma>1\). By definition of $g$ we also have that
	\begin{align}\label{linbox-eq3}
	g(z(\sigma\widetilde{L}))=\dfrac{\bar{a}(z(\sigma\widetilde{L}))+\bar{b}}{\bar{c}(z(\sigma\widetilde{L}))+\bar{d}}.
	\end{align}
	Then equating (\ref{linbox-eq2}) and (\ref{linbox-eq3}) we see that
	\begin{equation}\label{linbox-eq4}
	a^{2}(z(\sigma\widetilde{L}))+ab=\dfrac{\bar{a}(z(\sigma\widetilde{L}))+\bar{b}}{\bar{c}(z(\sigma\widetilde{L}))+\bar{d}}.
	\end{equation}
	By rewriting equation (\ref{linbox-eq4}), we get
	\begin{align}\label{linbox-eq5}
	a^{2}\bar{c}(z(\sigma\widetilde{L}))^{2}+(ab\bar{c}+a^{2}\bar{d})z(\sigma\widetilde{L})+ab\bar{d}=\bar{a}z(\sigma\widetilde{L})+\bar{b}.
	\end{align}
	Equation (\ref{linbox-eq5}) is a polynomial equation that holds for all the \(\sigma>1\) such that
	\(\sigma\widetilde{L}<\widetilde{L}<L\) and so, since there is an interval worth of such $\sigma$'s,
	either \(a^{2}=0\) or \(\bar{c}=0\). If \(a^{2}=0\), then \(g\) 
	would not be a homeomorphism for \(t<L\), which is impossible. If \(\bar{c}=0\), 
	then equation \(\left(\ref{linbox-eq5}\right)\), coupled with the fact that $\bar{a}\bar{d}=1$,
	implies that 
	\[
	a^{2}z(\sigma\widetilde{L})+ab=\bar{a}^2z(\sigma\widetilde{L})+\bar{a}\bar{b},
	\]
	and so \(g\left(t\right)=a^{2}t+ab\) for \(t\in\left(-\infty,M\right]\) for some 
	\(M>\widetilde{L}\), a contradiction to the definition of the breakpoint \(\widetilde{L}\). 
	Hence, in all cases we have a contradiction to the assumption that
	\(\widetilde{L}<\min\left\{L,g^{-1}\left(L\right)\right\}\) and so we have that
	\(\widetilde{L}\geq\min\left\{L,g^{-1}\left(L\right)\right\}\). The proof for the 
	final box is similar.
\end{proof}
\begin{rem}
	We notice that Lemma \ref{monod-initialbox} also holds for 
	\(y,z\in H^{<}\left(-\infty,p\right)\), by just applying its statement to \(y^{-1}\) 
	and \(z^{-1}\).
\end{rem}

\subsection{Building a Candidate Conjugator}
In this subsection, we prove several lemmas which show how to buid a conjugator, 
if it exists. If this is the case, then we prove that the conjugator must be unique. 
Given two elements \(y,z\in H\), the set of their conjugators is a coset of the
centralizer of either \(y\) or \(z\). Thus, it is important to begin by obtaining
properties of centralizers, which we will do next. After that, we will identify 
\(y\) and \(z\) inside a box close to the initial box using a suitable conjugator, 
as mentioned before. Then we repeat this process and build more pieces of this 
potential conjugator until we reach the final affinity box. We omit the proof
of some of the Lemmas, since they follow word-by-word from the original ones in
\cite{KassMat2012} with a slight adaptation in which we use the initial germs.
The proofs of the following two results are the same as that of \cite[Lemma 4.4]{KassMat2012}
and \cite[Corollary 4.5]{KassMat2012}.

\begin{lem}\label{monod-centralizer-identity}
	Let \(z\in H\) and suppose that there exist real numbers \(\lambda\) and \(\mu\) 
	satisfying \(-\infty<\lambda\leq\mu<+\infty\), \(z\left(t\right)\leq \lambda\), 
	for all \(t\in\left(-\infty, \mu\right]\) and that there is \(g\in H\) so
	that \(g\left(t\right)=t\) for all \(t\in \left(-\infty, \lambda\right]\) and 
	\(g^{-1}zg\left(t\right)=z\left(t\right)\) for each 
	\(t\in \left(-\infty, \mu\right]\). Then \(g\) is the identity map up to \(\mu\).
\end{lem}
%\begin{proof}
%	Notice that we can rewrite the equation
%	\(g^{-1}zg\left(t\right)=z\left(t\right)\) as 
%	\(g\left(t\right)=z^{-1}gz\left(t\right)\), for each
%	\(t\in \left(-\infty, \mu\right]\). By hypothesis, \(z\left(t\right)\leq \lambda\) for
%	all \(t\in\left(-\infty,\mu\right]\) and \(g\left(t\right)=t\) for \(t\leq\lambda\).
%	Then
%	\[g\left(t\right)=z^{-1}g\left(z\left(t\right) \right)=z^{-1}z\left(t\right),~\forall~t\in\left(-\infty,\mu\right].\]
%	Then \(g\left(t\right)=t\) for each \(t\leq\mu\),
%	as desired.
%\end{proof}

In case of \(z\in H^{<}\), the previous lemma yields the following consequence.

\begin{cor}\label{g=id}
	Let \(z\in H^{<}\) and \(g\in H\) such that \(g_{-\infty}=\left(1,0\right)\)	
	and \(g^{-1}zg=z\). Then \(g\) is the identity map.
\end{cor}
%\begin{proof}
%	Since \(g_{-\infty}=\left(1,0\right)\), we have there exists a number
%	\(L\in\mathbb{R}\) such that \(g\left(t\right)=t\) for all
%	\(t\in\left(-\infty, L\right]\). Applying the last Lemma several times, we get 
%	\(g\left(t\right)=t\) for all 
%	\(t\in \left(-\infty, z^{-k}\left(L\right)\right]\). Since \(z\in H^{<}\), we 
%	have \(\lim\limits_{k\rightarrow+\infty}z^{-k}\left(L\right)=+\infty\). 
%	Then \(g\) is the identity map.
%\end{proof}

The preceding two results allow us to construct a group monomorphism 
between the group of centralizers of elements from \(H\) 
%in particular from \(H^{<}\) 
and the group \(\mathrm{Aff}\left(\mathbb{R}\right)\) as well as showing uniqueness
of conjugators
\begin{lem}\label{monod-centralizers}
	Given \(z\in H^{<}\), the following map
	\begin{align*}
	\varphi_{z}\colon C_{H}\left(z\right) &\rightarrow \mathrm{Aff}\left(\mathbb{R}\right)\\
	g &\mapsto g_{-\infty},	
	\end{align*}
	\noindent is a group monomorphism.
\end{lem}
\begin{proof}
	First of all, for each \(g_{1},g_{2}\in C_{H}\left(z\right)\), with
	\((g_1)_{-\infty}=\left(a_{0}^{2},a_{0}b_{0}\right)\) and
	\((g_2)_{-\infty}=\left(\bar{a}_{0}^{2},\bar{a}_{0}\bar{b}_{0}\right)\),
	there exists \(L\in \mathbb{R}\) so that 
	\(g_{1}g_{2}\left(t\right)=a_{0}^{2}\bar{a}_{0}^{2}t+a_{0}^{2}\bar{a}_{0}\bar{b}_{0}+a_{0}b_{0}\)
	on $(-\infty,L]$.
	Then 
	\[
	(g_{1}g_{2})_{-\infty}=\left(a_{0}^{2}\bar{a}_{0}^{2},a_{0}b_{0}+a_{0}^{2}\bar{a}_{0}\bar{b}_{0}\right)=
	\left(a_{0}^{2},a_{0}b_{0}\right) \cdot \left(\bar{a}_{0}^{2},\bar{a}_{0}\bar{b}_{0}\right)=
	(g_1)_{-\infty} (g_2)_{-\infty}.
	\]
	so that \(\varphi_{z}\) is a well-defined group homomorphism.
	To show injectivity, suppose that
	\(\varphi\left(g_{1}\right)=\varphi\left(g_{2}\right)\) for $g_1,g_2 \in C_H(z)$. 
	Then
	\(\left(a_{0}^{2},a_{0}b_{0}\right)=\left(\bar{a}_{0}^{2},\bar{a}_{0}\bar{b}_{0}\right)\).
	Thus there exists a number \(L\in\mathbb{R}\) so that 
	\(g_{1}\left(t\right)=g_{2}\left(t\right)\) for all 
	\(t\in\left(-\infty,L\right]\). Let us define \(g\coloneqq g_{1}g_{2}^{-1}\). We 
	have \(g\left(t\right)=t\) for each \(t\in\left(-\infty,L\right]\). Moreover, we have
	 \(g^{-1}zg=z\). It follows from Corollary \ref{g=id} that 
	\(g\left(t\right)=t\) for all \(t\in\mathbb{R}\). Which implies that 
	\(g_{1}\left(t\right)=g_{2}\left(t\right)\) for each \(t\in\mathbb{R}\). 
	Therefore \(\varphi_{z}\) is a monomorphism.
\end{proof}
%\begin{rem}
%	We can consider \(z\in H^{<}\) in the previous result. 
%\end{rem}

\begin{prop}[Uniqueness]\label{monod-uniqueness}
	Let \(y,z\in H^{<}\) and \(g\in H\) be maps so that \(y^{g}=z\). Then the 
	conjugator \(g\) is uniquely determined by its initial germ \(g_{-\infty}\).
\end{prop}
\begin{proof}
	Let us assume that there are \(g_{1},g_{2}\in H\) so that \(g_{1}^{-1}yg_{1}=z\) 
	and \(g_{2}^{-1}yg_{2}=z\) and with the same initial germ. Then
	\(\left(g_{1}g_{2}^{-1}\right)^{-1}y\left(g_{1}g_{2}^{-1}\right)=y\).
	Defining \(g\coloneqq g_{1}g_{2}^{-1}\), we get that \(g\left(t\right)=t\) for 
	all \(t\in\left(-\infty,L\right]\),	which implies that the initial germ of \(g\) 
	is \(g_{-\infty}=\left(1,0\right)\). By Corollary \ref{g=id}, the unique centralizer of \(y\) with initial germ $\left(1,0\right)$
	is the identity map. Then \(g\left(t\right)=t\) for all \(t\in\mathbb{R}\). 
	Therefore, \(g_{1}=g_{2}\), which proves the uniqueness of a conjugator with a 
	given initial germ, if it exists.
\end{proof}

The next lemma gives a tool to identify the graphs of \(y\) and \(z\) inside 
suitable boxes via some candidate conjugator.

\begin{lem}[Identification Lemma]\label{monod-idlemma}
	Let \(y,z\in H^{<}\) and \(L\in\mathbb{R}\) be such that 
	\(y\left(t\right)=z\left(t\right)\) for all \(t\in\left(-\infty,L\right]\). Then 
	there exists \(g\in H\) so that \(z\left(t\right)=g^{-1}yg\left(t\right)\) for 
	every \(t\in\left(-\infty,z^{-1}\left(L\right)\right]\) and 
	\(g\left(t\right)=t\) in \(\left(-\infty,L\right]\). Moreover, this
	element \(g\) is uniquely determined on \(\left(L,z^{-1}\left(L\right)\right]\).
\end{lem}
\begin{proof}
We start showing that, if such a \(g\in H\) exists, then it is uniquely determined on \(\left(L,z^{-1}\left(L\right)\right]\). In fact, if such a \(g\in H\) exists then, for each
	\(t\in\left(L,z^{-1}\left(L\right)\right]\), we have
	\(y\left(g\left(t\right)\right)=g\left(z\left(t\right)\right)=z\left(t\right)\),
	since \(z\left(t\right)\leq L\). Therefore \(g\left(t\right)=y^{-1}z\left(t\right)\),
	for every \(t\in\left(L,z^{-1}\left(L\right)\right]\). 
	
	To show existence, we just define
	\[g\left(t\right)\coloneqq\begin{dcases}
	t,&\;\text{if}\;t\in\left(-\infty,L\right]\\
	y^{-1}z\left(t\right),&\;\text{if}\;t\in\left[L,z^{-1}\left(L\right)\right]
	\end{dcases}\]
	and we extend it to the real line from the point $\left(z^{-1}\left(L\right),y^{-1}\left(L\right)\right)$, by gluing some order-preserving affine map defined on $[z^{-1}\left(L\right),+\infty)$. We also define $g(\pm \infty)=\pm \infty$.
\end{proof}

	We repeatedly apply Lemma \ref{monod-idlemma} so that, if we iterate it $N$ times, we can
	build \(g\) on \(\left(-\infty, z^{-N}\left(L\right)\right]\) and this will be the key
	step for the Stair Algorithm in the next subsection. We conclude this subsection with a result whose proof can be obtained word for word from \cite[Lemma 4.13]{KassMat2012}.

\begin{lem}%[Conjugator for Powers]
\label{power-conj-lemma}
	Let \(y,z\in H^{<}\). Let us consider \(g\in H\) and \(n\in\mathbb{Z}_{>0}\).
	Then \(y^{g}=z\) if, and only if, \(\left(y^{n}\right)^{g}=z^{n}\).
\end{lem}

%\begin{proof}
%	If \(y^{g}=z\), then it follows easily that \((y^{n})^{g}=z^{n}\) by taking the
%	power \(n\) of the first equation. Reciprocally, if we define 
%	\(f\coloneqq (y^{n})^{g}=z^{n}\), we have \(y^{g}\) and \(z\) centralize \(f\). 
%	Let us suppose \((y^{g})_{-\infty}\neq z_{-\infty}\). Then either the first 
%	entries or the second ones of \((y^{g})_{-\infty}\) and \(z_{-\infty}\) are 
%	different. In both cases, \((y^{n})^{g}\neq z^{n}\), a contradiction. Then 
%	\((y^{g})_{-\infty}=z_{-\infty}\). By Lemma \ref{monod-centralizers}, we have
%	\(y^{g}=z\).
%\end{proof}

\subsection{The Stair Algorithm for \(H\)}

 We now adapt to $H$ the Stair Algorithm from \cite{KassMat2012} which constructs the unique candidate conjugator between two elements $y,z \in H$ with a given initial germ $(a^2,ab)\in \mathrm{Aff}\left(\mathbb{R}\right)$, that is, an element $g \in H$ such that, if there exists an $h \in H$ so that $h_{-\infty}=(a^2,ab)$ and $y^g=z$, then $h=g$.
 
	\begin{thm}[Stair Algorithm]\label{monod-stair}
		Let \(y,z\in H^{<}\) and let \(\left(-\infty,L\right]^{2}\) be the initial 
		box given by $y$ and $z$. Let us consider  
		\(\left(a^{2},ab\right)\in\mathrm{Aff}\left(\mathbb{R}\right)\) so that
		$a^2L+ab \le L$.
		Then there exists \(N\in\mathbb{Z}_{>0}\) such that the unique candidate 
		conjugator \(g \in H\) between \(y\) and \(z\) with initial germ 
		\(g_{-\infty}=\left(a^{2},ab\right)\) is given by
		\[g\left(t\right)=y^{-N}g_{0}z^{N}\left(t\right),\;\text{for  }\;t\in\left(-\infty,z^{-N}\left(L\right)\right],\]
		\noindent and affine otherwise, where \(g_{0}\in H\) is an arbitrary homeomorphism
		which is affine on \(\left(-\infty,L\right]^{2}\) and so that 
		\((g_0)_{-\infty}=\left(a^{2},ab\right)\).
	\end{thm}
	
	%{\large \francesco{2020-02-21: the following paragraph, Corollary and Remark are all new}}
	
	\begin{rem}
	We observe that the hypothesis on $(a^{2},ab)$ is a mild one. It ensures that $g_0(L)\le L$ and so, 
	    up to replacing \(g_{0}\) by \(g_{0}^{-1}\) and 
		switching the role of \(y\) and \(z\), we can always assume that \(a^2L+ab \le L\).
	\end{rem}
	
	Before giving the proof of Theorem \ref{monod-stair}, we observe the following corollary and make a comment 
	about completely characterizing conjugacy in Monod's group $H$.
	
	\begin{cor}\label{thm:characterize-conjugacy}
	Let \(y,z\in H^{<}\) and let \(\left(-\infty,L\right]^{2}\) 
	and \(\left[R,+\infty\right)^{2}\) be, respectively, 
	the initial and the final box given by $y$ and $z$. There is a $g \in H$ such that $y^g=z$
		if and only if there is some $(a^2,ab) \in\mathrm{Aff}\left(\mathbb{R}\right)$
		so that
		$a^2L+ab \le L$ and
		\[
		\lim_{N \to \infty}y^{-N}g_{0}z^{N}\left(t\right)
		\]
		is affine inside \(\left[R,+\infty\right)^{2}\) and
		where \(g_{0}\in H\) is an arbitrary homeomorphism
		which is affine on \(\left(-\infty,L\right]^{2}\) and so that 
		\((g_0)_{-\infty}=\left(a^{2},ab\right)\).
	\end{cor}
	
	\begin{rem}
	Corollary \ref{thm:characterize-conjugacy} gives a characterization of conjugacy
	inside Monod's group $H$. However, as is, such characterization cannot be
	used to construct a finite set of candidate conjugators and thus
	use them to solve the conjugacy problem in Monod's group $H(A)$ for a suitable
	subring $A \subseteq \mathbb{R}$ in a manner analogous what was done for the Thompson-Stein groups $\mathrm{PL}_{A,G}([0,1])$ in \cite{KassMat2012} for a suitable
	subring $A \subseteq \mathbb{R}$ and subgroup $G \subseteq U(A)_{>0}$ of the group
	of the positive units of $A$.
	We now explain better why not.
	Lemma 5.4 in \cite{KassMat2012} shows that there are only finitely conjugators between $y$ and $z$ whose initial and final slope lie within a bounded interval. 
	After having used a suitable isomorphism (Lemma \ref{thm:unbounded-to-bounded} below) and thus considering a version of $H$ over the the unit interval $[0,1]$, we prove in Lemma
	\ref{thm:generalized-KasMat-5.3} below a similar result for centralizers with first and second derivative
	lying within bounded intervals. Even if Lemma \ref{thm:generalized-KasMat-5.3} can indeed be generalized to
	study conjugators with bounded first and second derivative (to get a result analogous to Lemma 5.4 in \cite{KassMat2012}), we cannot use
	the same trick of Lemmas 7.1 and 7.2 in \cite{KassMat2012} to bound both
	derivatives. By replacing a conjugator $g$ with $y^ng$ we can only bound the first derivative (so that it lives in a bounded interval), but have no available bound on the second derivative appearing in Lemma \ref{thm:generalized-KasMat-5.3} below: in other words, we can bound the $a$ appearing in $g_{-\infty}=(a^2,ab)$,
	but not the $b$ and so we would have to test continuum many $b$'s (or equivalently,
	continuum many initial germs) to find candidate conjugators.
	\end{rem}
	
	\begin{proof}[Proof of Theorem \ref{monod-stair}]
		First of all, we notice that we will consider \(y,z\in H^{<}\) such 
		that their initial germs are in the same conjugacy class in
		\(\mathrm{Aff}\left(\mathbb{R}\right)\), otherwise \(y\) and \(z\) 
		cannot be conjugate to each other by Lemma \ref{monod-germsconj}. 	
		Moreover, we further assume that
		\(\left(a^{2},ab\right)\in\mathrm{Aff}\left(\mathbb{R}\right)\)
		conjugates \(y_{-\infty}\) to \(z_{-\infty}\) in 
		\(\mathrm{Aff}\left(\mathbb{R}\right)\), otherwise there cannot be a conjugator 
		\(g\) for \(y\) and \(z\) with initial germ
		\(g_{-\infty}=\left(a^{2},ab\right)\), again by Lemma \ref{monod-germsconj}.
		Now, let \(\left[R,+\infty\right)^{2}\) be the final box and let
		\(N\in\mathbb{Z}_{>0}\) be sufficiently large so that 
		\(\min\left\{z^{-N}\left(L\right),y^{-N}\left(a^{2}L+ab\right)\right\}>R\).
		
		We will now build a candidate conjugator $g$ between $y^{N}$ and $z^{N}$ as the product of two functions \(g_{0}\) and \(g_{1}\) and then use
		Lemma \ref{power-conj-lemma}. Since $y,z \in H^{<}$, a direct calculation shows that \(y^{N}\) and \(z^{N}\) are affine in the initial and final boxes of $y$ and $z$, so we can take them as the initial and final boxes of \(y^{N}\) and \(z^{N}\). %Then by Lemma \ref{monod-initialbox} \(g\) must be affine on \(\left(-\infty,L\right]^{2}\). 
		We define %an ``approximate conjugator'' 
		\(g_{0}\) as \(g_{0}\left(t\right)\coloneqq a^{2}t+ab\),
		on \(\left(-\infty,L\right]^2\) and extend it to the real line so
		that \(g_{0}\in H\). Our assumption on $(a^2,ab)$ ensures that $g_0(L)\le L$.
		Now we 
		define \(y_{1}\coloneqq g_{0}^{-1}yg_{0}\) and we look for a conjugator 
		\(g_{1}\) between \(y_{1}^{N}\) and \(z^{N}\). We remark that 
		\(y_{1}^{N}\) and \(z^{N}\) coincide on	\(\left(-\infty,L\right]\), since
		\(y_{1}^{N}=g_{0}^{-1}y^{N}g_{0}=z^{N}\) and $y^N,z^N \in H^{<}$ are affine on \(\left(-\infty,L\right]\).

		Making use of the Identification Lemma \ref{monod-idlemma}, we define
		a $g_{1}\in H$ so that
		\[g_{1}\left(t\right)\coloneqq\begin{dcases}
		t,&\;\text{if}\;t\in\left(-\infty,L\right]\\
		y_{1}^{-N}z^{N}\left(t\right),&\;\text{if}\;t\in\left[L,z^{-N}\left(L\right)\right].
		\end{dcases}\]
		\noindent By construction we have \(g_{1}^{-1}y_{1}^{N}g_{1}=z^{N}\)	on \(\left(-\infty,z^{-N}\left(L\right)\right]\). We now construct a function  
		\(g\) on \(\left(-\infty,z^{-N}\left(L\right)\right]\) by defining
		\(g\left(t\right)\coloneqq g_{0}g_{1}\left(t\right)\),
		for \(t\in \left(-\infty,z^{-N}\left(L\right)\right]\). We observe that 
		the last part of \(g\) is defined inside in the final box since
		\(t=z^{-N}\left(L\right)>R\) and
		\[g\left(z^{-N}\left(L\right)\right)=g_{0}g_{1}\left(z^{-N}\left(L\right)\right)>R.\]
		\noindent Moreover, by construction, \(g\) is a conjugator for \(y^{N}\) 
		and \(z^{N}\) on \(\left(-\infty,z^{-N}\left(L\right)\right]\), that is, 
		\(g=y^{-N}gz^{N}\) on \(\left(-\infty,z^{-N}\left(L\right)\right]\). Therefore,
		\[g\left(t\right)=y^{-N}gz^{N}\left(t\right)=y^{-N}g_{0}g_{1}z^{N}\left(t\right)=y^{-N}g_{0}z^{N}\left(t\right),\]
		\noindent since \(g_{1}z^{N}\left(t\right)=z^{N}\left(t\right)\) for every
		\(t\in \left(-\infty,z^{-N}\left(L\right)\right]\).
		
		If \(g\) is not an affine M\"obius function on \(\left[R,z^{-N}(L)\right]\), then \(g\) cannot be extended
		to a conjugator of \(y^N\) and \(z^N\) and the uniqueness of the shape of 
		\(g\) (Proposition \ref{monod-uniqueness}) says that continuing the Stair 
		Algorithm will build a function that cannot be a conjugator and therefore, 
		a conjugator with initial germ \(\left(a^{2},ab\right)\) cannot exist or 
		it would coincide with \(g\) on \(\left(-\infty, z^{-N}\left(L\right)\right]\).
		In the case that \(g\) is a affine M\"obius function on \(\left[R,z^{-N}\left(L\right)\right]\), we extend \(g\) to the whole real 
		line by extending its affine piece on
		\(\left[R,z^{-N}\left(L\right)\right]\). The map that we construct (which we still call \(g\)) lies 
		in \(H\).
		
		By Lemma \ref{monod-initialbox} and Proposition \ref{monod-uniqueness}, if 
		there exists a conjugator between \(y^{N}\) and \(z^{N}\), with initial 
		germ \(\left(a^{2},ab\right)\), it must be equal to \(g\). Then we just 
		check if \(g\) conjugates \(y^{N}\) to \(z^{N}\). If \(g\) conjugates 
		\(y^{N}\) to \(z^{N}\) then, by Lemma \ref{power-conj-lemma}, \(g\) is a 
		conjugator between \(y\) and \(z\), as desired.
	\end{proof}
	\begin{rem}\label{thm:switch-<->}
		Let us suppose that \(y,z\in H^{<}\cup H^{>}\).	In order to be conjugate, the 
		Lemma \ref{monod-germsconj} says that their initial germ must be in the same 
		conjugacy class in \(\mathrm{Aff}\left(\mathbb{R}\right)\). Similarly,
		their final germs must be in the same conjugacy class in 
		\(\mathrm{Aff}\left(\mathbb{R}\right)\). In other words, either both \(y\) 
		and \(z\) are in \(H^{<}\) or both are in \(H^{>}\). Furthermore, since 
		\(g^{-1}yg=z\) if and only if \(g^{-1}y^{-1}g=z^{-1}\), we can reduce
		the study to the case where they are both in \(H^{<}\).
	\end{rem}
	\begin{rem}
		The stair algorithm for \(H^{<}\) can be reversed. This means that we can 
		apply it in order to build a candidate for a conjugator between 
		\(y,z\in H^{>}\). Thus, given an element
		\(\left(a^{2},ab\right)\in\mathrm{Aff}\left(\mathbb{R}\right)\), we can 
		determine whether or not there is a conjugator \(g\) with final 
		germ \(g_{+\infty}=\left(a^{2},ab\right)\). The proof is similar.
		We just begin to construct \(g\) from the final box.
	\end{rem}
	We observe that the proof of Stair Algorithm does not depend on the choice of 
	\(g_{0}\), the only requirement on it is that it be affine on the initial box and
	\(g_{0_{-\infty}}=\left(a^{2},ab\right)\). Moreover, it gives a way to find 
	candidate conjugators, if they exist, and we have chosen an initial germ. 
	
	The following are two examples of construction of candidate conjugators
	via the Stair Algorithm. In the first example the candidate is indeed a conjugator,
	while in the second it is not.

    \begin{ex}\label{ex:Stair-Working}
    	Consider the maps $y\left(t\right)=t-1$ and
    	\[z\left(t\right)=\begin{dcases}
    	\dfrac{2t-2}{\frac{-3}{2}t+2},&\;\text{if}\;t\in\left[0,1\right];\\
    	\dfrac{-2t+2}{\frac{-3}{2}t+1},&\;\text{if}\;t\in\left[1,2\right];\\
    	t-1,&\;\text{otherwise}.
    	\end{dcases}\]
    Notice that $y,z\in H^{<}$ and that their initial and final germs are equal. Moreover, we have \(L=0\)
    and \(R=2\).
    Now we take
    $\left(1,-1\right)\in\mathrm{Aff}\left(\mathbb{R}\right)$
    and construct a candidate conjugator between \(y^{4}\) and
    \(z^{4}\). We follow the procedure of the proof of the Stair Algorithm  
    and define the maps $g_{0}\left(t\right)\coloneqq t-1$ and
    \[g_{1}\left(t\right)\coloneqq \begin{dcases}
    \dfrac{\frac{1}{2}t}{-\frac{3}{2}t+2},&\;\text{if}\;t\in\left[0,1\right];\\
    t,&\;\text{otherwise}.
    \end{dcases}\]
    We then define $g:=g_{0}g_{1}$ and see that
    \[g\left(t\right)\coloneqq\begin{dcases}
    \dfrac{2t-2}{\frac{-3}{2}t+2},&\;\text{if}\;t\in\left[0,1\right];\\
    t-1,&\;\text{otherwise}.
    \end{dcases}\quad\textrm{and}\quad g^{-1}\left(t\right)=\begin{dcases}
    \dfrac{2t+2}{\frac{3}{2}t+2},&\;\text{if}\;t\in\left[-1,0\right];\\
    t+1,&\;\text{otherwise}.
    \end{dcases}\]
    We notice that \(g\in H\). A direct calculation shows that $g$ conjugates \(y^{4}\) to \(z^{4}\). By Lemma \ref{power-conj-lemma}, the element
    \(g\) is a conjugator between \(y\) and \(z\).
    \end{ex}
    
    \begin{ex}\label{Stair-Not-Working}
	Consider the maps $y\left(t\right)=t-1$ and 
	\begin{displaymath}
	z\left(t\right)=\begin{dcases}
	\dfrac{-2t+2}{\frac{-3}{2}t+1},&\;\text{if}\;t\in\left[1,2\right];\\
	t-1,&\;\text{otherwise}.
	\end{dcases}
	\end{displaymath}
	Notice that $y,z\in H^{<}$ and that their initial and final germs are equal. We observe that \(L=1\)
    and \(R=2\).
    Now we take 
    \(\left(1,0\right)\in\mathrm{Aff}\left(\mathbb{R}\right)\)
    and construct a candidate conjugator between \(y^{3}\) and 
    \(z^{3}\). We follow the procedure of the proof of the Stair Algorithm  
    and define the maps $g_{0}\left(t\right)=t$ and
	\[g_{1}\left(t\right)=\begin{dcases}
	\dfrac{-\frac{7}{2}t+3}{-\frac{3}{2}t+1},&\;\text{if}\;t\in\left[1,2\right];\\
	\dfrac{-5t+9}{\frac{-3}{2}t+\frac{5}{2}},&\;\text{if}\;t\in\left[2,3\right];\\
	t,&\;\text{otherwise}.
	\end{dcases}\]
	We then define $g:=g_0g_1$ and see that
	\[g\left(t\right)=\begin{dcases}
	\dfrac{-\frac{7}{2}t+3}{-\frac{3}{2}t+1},&\;\text{if}\;t\in\left[1,2\right];\\
	\dfrac{-5t+9}{\frac{-3}{2}t+\frac{5}{2}},&\;\text{if}\;t\in\left[2,3\right];\\
	t,&\;\text{otherwise}.
	\end{dcases}\]
	\noindent We notice that $g$ is not a linear M\"{o}bius 
	function on $\left[2,3\right]$. Thus, by Theorem \ref{monod-stair}, the element $g$ cannot be a 
	conjugator between $y^{3}$ and $z^{3}$. By Lemma 
	\ref{power-conj-lemma}, $g$ cannot be a conjugator between 
	$y$ and $z$ as well.
\end{ex}

\begin{rem}\label{rem:subfied}
    Although this section is stated for $H$ for the sake of consistency of the paper, the proofs that all the results of this section hold for $H(A)$ too, with the following provisions:
    \begin{enumerate}
        \item The elements $L$ and $R$ defined
    for the affinity boxes in Lemma \ref{monod-initialbox} must live in $A$. This can always be achieved since, given any initial box $(-\infty,L]^2$, we can then take an
    $L' \le L$ in $L \in A$ and consider the box
    $(-\infty,L']^2$. Similarly we can do for the final one.
        \item Lemma \ref{monod-germsconj} needs to be stated by saying that $y_{-\infty}^{\mathrm{Aff}(A)}=z_{-\infty}^{\mathrm{Aff}(A)}$, where the affine
        group of $A$ is the subgroup of $\mathrm{Aff}(\mathbb{R})$ 
        defined by $\mathrm{Aff}(A)= (U(A)_{>0},\cdot)\ltimes (A,+)$, where $U(A)_{>0}$
        is the group of the positive units of $A$. Similarly, we must have
        $y_{+\infty}^{\mathrm{Aff}(A)}=z_{+\infty}^{\mathrm{Aff}(A)}$.
        %\item Since the inverse of a map $g(t)=a^2t+ab$ is $g^{-1}(t)=a^{-2}t-a^{-1}b$,
        %we need to require that the subring $A$ is actually a subfield of $\mathbb{R}$.
    \end{enumerate}
    
    \end{rem}

\section{The Mather invariant}\label{sec:mather}

	We now construct a conjugacy invariant for a class of functions, called \emph{Mather invariant}, by 
	adapting ideas from \cite{Mather1974,Mat2010}. While in the previous section we worked with $y,z \in H^{<}$, in this section we will work with $y,z \in H^{>}$ as it helps with the arguments and we can do so without loss of generality because of
	Remark \ref{thm:switch-<->}. We construct such invariant to deal with the case
	$y'(\pm \infty)=z'(\pm \infty)=1$ where the point of view of the Stair Algorithm cannot be used to cover all cases when computing element centralizers of elements which will be studied in the next section.

	In the remainder of this section 
	we assume \(y,z\in H^{>}\) such that \(y\left(t\right)=z\left(t\right)=t+b_{0}\) 
	if \(t\in \left(-\infty,L\right]\) and
	\(z\left(t\right)=y\left(t\right)=t+b_{1}\) if \(t\in \left[R,+\infty\right)\), 
	for some suitable \(b_{0},b_{1}>0\), where \(L\) and \(R\) 
	are, respectively, sufficiently large negative and
	positive real numbers.

	Let \(N\in\mathbb{Z}_{>0}\) be large enough so that
	\[
	y^{N}\left(\left(y^{-1}\left(L\right),L\right)\right) \cup z^{N}\left(\left(z^{-1}\left(L\right),L\right)\right)\subset \left(R,+\infty\right).
	\]
	We intend to find a map \(s\in H\) such that
	\(s\left(y^{k}\left(L\right)\right)=k,\) for every \(k\in\mathbb{Z}\).
	We thus define the map $s$ as
	\begin{align*}
		s\colon \mathbb{R}&\rightarrow \mathbb{R}\\
		t&\longmapsto s\left(t\right)\coloneqq \begin{dcases}
		s_{-1}\left(t\right),\;&\text{if}\;t\in(-\infty,L]\\
		s_{j}\left(t\right),\;&\text{if}\;t\in[y^{j}\left(L\right),y^{j+1}\left(L\right)]\\
		s_{N-1}\left(t\right),\;&\text{if}\;t\in[y^{N-1}\left(L\right),+\infty)
		\end{dcases}
	\end{align*}
	\noindent where
	\begin{align*}
		s_{-1}\colon \left[y^{-1}\left(L\right),L\right]&\rightarrow \left[-1,0\right] &s_{N-1}\colon [y^{N-1}\left(L\right),y^{N}\left(L\right)]&\rightarrow [N-1,N]\\
		t &\mapsto \dfrac{t-L}{b_{0}},\qquad &t &\mapsto \dfrac{t-y^{N-1}\left(L\right)}{b_{1}}+N-1
	\end{align*}
	and
	\begin{align*}
		s_{j}\colon [y^{j}\left(L\right),y^{j+1}\left(L\right)]&\rightarrow [j,j+1]\\
		t &\mapsto \dfrac{t-y^{j}\left(L\right)}{y^{j+1}\left(L\right)-y^{j}\left(L\right)}+j,\;\forall\;j=0,1,,\ldots,N-2.
	\end{align*}
	
	Since \(L\) is a fixed point of some hyperbolic element from 
	\(\mathrm{PSL}_{2}\left(\mathbb{R}\right)\), so is \(y^{j}\left(L\right)\), for 
	every \(j=0,1,\ldots,N-2,N-1\). Also, since all of the $s_i$'s are affine with strictly
	positive slope, they can all be written as $s_i(t)=a_i^2 t + a_i b_i$ for suitable $a_i,b_i \in \mathbb{R}$
	and so  \(s\in H\). Moreover, it is clear that 
	\(s\left(y^{k}\left(L\right)\right)=k\), for all \(k\in\mathbb{Z}\).
	If we define \(\overline{y}\coloneqq sys^{-1}\) and \(\overline{z}\coloneqq szs^{-1}\),
	we get that both functions are well-defined and lie in \(H\). Now, we observe that
	\begin{itemize}
		\item[(i)] If 
		\(t\in \left(-\infty,0\right]\cup\left[N-1,+\infty\right)\), then
		\(\overline{y}\left(t\right)=\overline{z}\left(t\right)=t+1\);
		\item[(ii)] \(\overline{y}^{N},\overline{z}^{N}\in H\).
	\end{itemize}
	
	We define the circles
	\[C_{0}=\left(-\infty,0\right]/\{t\sim t+1\}\quad\text{and}\quad C_{1}=\left[N-1,+\infty\right)/\{t\sim t+1\}.\]
	Let us consider the natural projections
	\(p_{0}\colon\left(-\infty, 0\right]\rightarrow C_{0}\) and \\ 
	\(p_{1}\colon\left[N-1, +\infty\right)\rightarrow C_{1}\).
	\noindent Then we restrict \(\overline{y}^{N}\) to the interval 
	\(\left[-1,0\right]\) such that \(p_{0}\) surjects it onto \(C_{0}\). 
	Since \(N\) is sufficiently large so that 
	\(\overline{y}^{N}\left(\left(\overline{y}^{-1}\left(L\right),L\right)\right)\subset \left[R,+\infty\right)\), 
	it follows that \(\overline{y}^{N}\) maps \(\left[-1,0\right]\) to 
	\(\left[R,+\infty\right)\). Passing to quotients, we define 
	\(\overline{y}^{\infty}\colon C_{0}\rightarrow C_{1}\) such as 
	\(\overline{y}^{\infty}\left(\left[t\right]\right)=\left[\overline{y}^{N}\left(t\right)\right]\) 
	making the following diagram commutative
	\[\xymatrix{
		\ar @{} [dr] |{\circlearrowright}
		\left[-1, 0\right] \ar[d]_-{p_{0}} \ar[r]^-{\overline{y}^{N}} & \left[N-1,+\infty\right) \ar[d]^-{p_{1}} 
		\\C_{0} \ar@{-->}[r]_{\overline{y}^{\infty}} & C_{1}
	}\]
	We emphasize that the map \(\overline{y}^{\infty}\) does not depend on the 
	specific chosen value of \(N\), since if \(m\geq N\),
	\(\overline{y}^{m}\left(t\right)=\overline{y}^{m-N}\left(\overline{y}^{N}\left(t\right)\right)\),
	where \(\overline{y}^{N}\left(t\right)\in \left(R,+\infty\right)\) and
	\[\overline{y}^{N}\left(t\right)\sim \overline{y}^{N}\left(t\right)+1=\overline{y}\left(\overline{y}^{N}\left(t\right)\right)\sim\ldots\sim \overline{y}^{m-N}\left(\overline{y}^{N}\left(t\right)\right)=\overline{y}^{m}\left(t\right).\]
	Similarly, we define the map \(\overline{z}^{\infty}\). We remark that both these 
	maps are piecewise-M\"obius homeomorphisms from the circle \(C_{0}\) to the 
	circle \(C_{1}\). They are called the \textbf{Mather invariants} of 
	\(\overline{y}\) and \(\overline{z}\).

    Assume now that there exists a $g \in H$ such that $gz=yg$. By conjugating by $s$,  we get the equation \(\overline{gz}=\overline{y}\overline{g}\), where $\overline{g} \in H$. Since $\overline{y}$ and $\overline{z}$ are equal to the translation $t \mapsto t+1$ around $\pm \infty$, then the equation \(\overline{gz}=\overline{y}\overline{g}\) implies
    that $\overline{g} \in H$ is periodic 
    for real numbers that are sufficiently large positive and negative and so, around $-\infty$ where $\overline{g}(t)=a^2t+ab$ is affine, we must have that
    $a^2=1$ so $\overline{g}$ is a translation, otherwise $\overline{g}$ would not be periodic. Similarly, $\overline{g}$ is a translation around $+\infty$. Thus $g$ itself is a translation
    around $\pm \infty$. We record this observation for independent later use.
    
    \begin{rem}
    \label{thm:affine-commute-translation}
        If an affine map $g(t)=a^2t+ab$ commutes with a translation
        $z(t)=t+k$, then $a^2=1$.
    \end{rem}

	Now, going back to the argument above, we see that it induces the equation \(\overline{gz}^{N}=\overline{y}^{N}\overline{g}\), 
	where \(\overline{g}\) is periodic on 
	\(\left(-\infty,0\right)\cup\left(N-1,+\infty\right)\) since it commutes with $\overline{y}$ and $\overline{z}$ on such intervals, and so \(\overline{g}\) passes
	to quotients and becomes
	\begin{align}\label{rot}
		v_{1, m}\overline{z}^{\infty}=\overline{y}^{\infty}v_{0,\ell},
	\end{align}
	as done in \cite{BurMatVent2016,Mat2010}, since 
	\(\overline{g},\overline{y},\overline{z}\in H\) and
	\(v_{0,\ell}\coloneqq p_{0}\overline{g}\) and 
	\(v_{1, m}\coloneqq p_{1}\overline{g}\) are rotations of the circles \(C_{0}\)
	and \(C_{1}\), respectively, where \(\ell,m\) are the translation terms of \(g\) 
	on \(\left(-\infty,0\right)\) and \(\left(N-1+\infty\right)\), respectively.

	The proof of the next result shows the relation between the Stair Algorithm and the Mather invariant.

	\begin{thm}\label{matconj}
		Let \(y,z\in H^{>}\) be such that 
		\(y\left(t\right)=z\left(t\right)=t+b_{0}\) for 
		\(t\in(-\infty, L]\) and \(y\left(t\right)=z\left(t\right)=t+b_{1}\) for 
		\(t\in [R,+\infty)\) and let 
		\(\overline{y}^{\infty},\overline{z}^{\infty}\colon C_{0}\rightarrow C_{1}\)
		be the corresponding Mather invariants. Then \(y\) and \(z\) 
		are conjugate in \(H\) if and only if 
		\(\;\overline{y}^{\infty}\) and \(\;\overline{z}^{\infty}\) 
		differ by rotations \(v_{0,\ell}\) and \(v_{1,m}\) of the 
		domain and range circles, for some \(\ell,m\in\mathbb{R}\):
		\[\xymatrix{\ar @{} [dr] |{\circlearrowright}
			C_{0} \ar[d]_-{v_{0, \ell}} \ar[r]^-{\overline{z}^{\infty}} & C_{1} \ar[d]^-{v_{1, m}} \\
			C_{0} \ar[r]_{\overline{y}^{\infty}} & C_{1}
		}\]
	\end{thm}

	\begin{proof}
		The calculations above yield that, if \(g\in H\)
		conjugates \(y\) and \(z\), then equation 
		\(\left(\ref{rot}\right)\) is satisfied, which is equivalent to say that 
		\(\;\overline{y}^{\infty}\) and 
		\(\;\overline{z}^{\infty}\) differ by rotations \(v_{0,\ell}\) 
		and \(v_{1,m}\) of the domain and range circles, for some
		\(\ell,m\in\mathbb{R}\).
		
		Conversely, let us assume that there are \(\ell,m\in\mathbb{R}\)
		such that equation \(\left(\ref{rot}\right)\) is satisfied.
		Then, we choose \(g_{0}\in H\) which is affine in the initial 
		box with a initial germ \((g_0)_{-\infty}=\left(1,\ell\right)\).
		Then we define a map $g$ as the following pointwise limit
		\(g\left(t\right)\coloneqq \lim_{n\rightarrow+\infty} y^{n}g_{0}z^{-n}\left(t\right)\).
		By the Stair Algorithm Theorem \ref{monod-stair}, we have 
		\(gz=yg\), where \(g\in \mathrm{Homeo_+}(\mathbb{R})\) and it is such that, on any bounded interval, it coincides with the restriction of some function 
		$\mathrm{PPSL}(\mathbb{R})$ to such interval. We need to 
		show that \(g\in H\). By construction, \(g\) has finitely many 
		breakpoints in \(\left(-\infty,N-1\right]\).
		Conjugating both sides of the equation \(gz=yg\) by \(s\), we 
		get
		\[\overline{gz}=\overline{yg}.\]
		\noindent For all
		\(t\in\left(-\infty,0\right]\cup\left[N-1,+\infty\right)\), 
		we have
		\(\overline{y}\left(t\right)=\overline{z}\left(t\right)=t+1\).
		Thus 
		\(\overline{g}\left(t+1\right)=\overline{g}\left(t\right)+1\),
		for each \(t\in\left(-\infty,0\right]\cup\left[N-1,+\infty\right)\) and we can pass to quotients. Moreover, as argued above during the definition of the Mather
		invariants, we have that $\overline{g}$ is a translation $\overline{g}(t)=t+\ell$ on $(-\infty,0]$, while we still need to show that $\overline{g}$ is a translation on $[N-1,+\infty)$. Up to switch the role of
		$\overline{y}$ and $\overline{z}$, we can assume that $\ell \ge 0$.
		Passing the equation \(\overline{g}\overline{z}^{N}=\overline{y}^{N}\overline{g}\)
		to quotients, we obtain
		\[\overline{g}_{ind}\overline{z}^{\infty}\left(\left[t\right]\right)=\overline{y}^{\infty}v_{0, \ell}\left(\left[t\right]\right),\]
		\noindent for a suitable well-defined \(\overline{g}_{ind}\).
		%which exists is well-defined since 
		%\(\overline{g}\left(t+1\right)=\overline{g}\left(t\right)+1\). 
		By equation 
		\(\left(\ref{rot}\right)\), we have
		\[\overline{g}_{ind}\overline{z}^{\infty}\left(\left[t\right]\right)=v_{1, m}\overline{z}^{\infty}\left(\left[t\right]\right),\]
		and so, by the cancellation law, we have
		\[\overline{g}_{ind}=v_{1, m}\]
		so that $\overline{g}_{ind}$ is a rotation by \(m\) of the 
		circle \(C_{1}\).
		%\francesco{The next phrasing and parts of the diagram are new \textbf{and delicate}. Please, read them and let me know if you agree. Notice, $N_0$ is possibly much bigger than $N$.} 
		We now choose $N_0 \ge N$ large enough so that $d:=\overline{z}^{N_0}(-1-\ell)\ge N-1$ so that
		\[
		\overline{z}^{N_0}([-1-\ell,-\ell])=
		\overline{z}^{N_0}([-1-\ell,z(-1-\ell)])=[d,z(d)]=[d,d+1],
		\]
		and
		\[\overline{g}(d)=\overline{y}^{N_0}(\overline{g}(-1-\ell))= \overline{y}^{N_0}(-1)\ge N-1.
		\]
		Hence
		the following commutative diagram holds
		\[\xymatrix{
			& & & \left[-1-\ell,-\ell\right] \ar[dd]_(.5){p_{0}}|(.29)\hole \ar[rr]^-{\overline{z}^{N_0}} \ar[dlll]_{\overline{g}} & & [d,+\infty) \ar[dd]^(.5){p_{1}} \ar[dlll]_{\overline{g}}\\
			\left[-1,0\right] \ar[dd]_-{p_{0}} \ar[rr]_(.5){\overline{y}^{N_0}} & & \left[N-1,+\infty\right) \ar[dd]^-{p_{1}} \\
			& & & C_{0} \ar[rr]^(.5){\overline{z}^{\infty}} \ar[dlll]_{v_{0,\ell}}|(.42)\hole & & C_{1} \ar[dlll]^{\overline{g}_{ind}=v_{1, m}} \\
			C_{0} \ar[rr]_{\overline{y}^{\infty}} & & C_{1}
		}\]
		To finish the proof, 
		we need to see that \(\overline{g}\) is a affine M\"obius map 
		on \(\left[N-1,+\infty\right)\), which will mean that 
		\(\overline{g} \in H\). From the previous commutative diagram 
		we get \(v_{1, m}p_{1}=p_{1}\overline{g}\), which implies that
		\(\left[\overline{g}\left(t\right)\right]=\left[t+m\right]\)
		for \([t]\in C_1\).
		By definition of the equivalence relation and the fact that 
		\(\overline{g}\) is a periodic continuous function, we have that there 
		exists some \(r\in\mathbb{Z}\) such that,
		\(\overline{g}\left(t\right)\coloneqq t+m+r\),
		for all \(t\in\left[N-1,+\infty\right)\). Therefore, 
		\(\overline{g} \in H\).
	\end{proof}

	%The important point to note here is that Mather invariant and Theorem 
	%\ref{matconj} give a conjugacy invariant for maps \(y,z \in H^{>}\) for functions
	%which are translations around \(\pm\infty\). Consequently, they can be used to 
	%give an obstruction to the success of the Stair Algorithm which builds 
	%candidate conjugators, which actually may never be conjugators in 
	%\(H\) if \(y\) and \(z\) are not conjugate.

\begin{rem}
    The results of this section rely on Lemma \ref{monod-transitive} and so,
    in order to generalize them to $H(A)$, one needs to prove a generalized version of
    Lemma \ref{monod-transitive} for $H(A)$. In particular, the 
    construction of the homeomorphism $s$ at the beginning of this section
    requires a version of Lemma \ref{monod-transitive} for
    $H(A)$ to construct the maps $s_j$ for $j=0,\ldots, N=2$. This is thus true for Section \ref{sec:centralizer} too.
    %the proofs do not immediately
    %carry over to $H(A)$ for a subring $A$ of $\mathbb{R}$. This is thus true for Section \ref{sec:centralizer} too. In particular, it is not straight that the construction of the homeomorphism $s$ at the beginning of this section works in general, since we would need
    %a result similar to
    %Lemma \ref{monod-transitive} for $H(A)$ to construct the maps $s_j$ for $j=0,\ldots, N=2$.
    \end{rem}

\section{Centralizers}\label{sec:centralizer}
In this section, we use the conjugacy tools we have just constructed to
calculate the centralizers of elements from 
$H$. We start by performing some calculations for centralizers 
of elements in \(\mathrm{Aff}\left(\mathbb{R}\right)\) and use this information to classify the centralizers of elements from $H$.

\subsection{Centralizers in \(\mathrm{Aff}(\mathbb{R})\)}\label{subsec:centralizers-affine-group}
%In this section, we deal with centralizers of elements
%from $\mathrm{Aff}\left(\mathbb{R}\right)$. 
Since Lemma 
\ref{monod-centralizers} gives a monomorphism 
$\varphi_{z}\colon C_{H}\left(z\right) \to \mathrm{Aff}\left(\mathbb{R}\right)$,
it 
makes sense to investigate centralizers in \(\mathrm{Aff}\left(\mathbb{R}\right)\).

If $\left(a,b\right)\in \mathrm{Aff}(\mathbb{R})$ and $a\neq 1$, then
$\left(c,d\right)\in C_{\mathrm{Aff}(\mathbb{R})}\left(a,b\right)$ if and only if
\[\left(c,d\right)=\left(a^{-1},-a^{-1}b\right)\left(c,d\right)\left(a,b\right)=\left(c,-a^{-1}b+a^{-1}d+a^{-1}bc\right)\]
which is equivalent to $d=\dfrac{b(c-1)}{a-1}$, and so
\[
C_{\mathrm{Aff}(\mathbb{R})}(a,b)=\left\{\left(c,\dfrac{b(c-1)}{a-1}\right)\in \mathrm{Aff}(\mathbb{R})\; \bigg \vert \; c\in \mathbb{R}_{>0}\right\}\cong
	\left(\mathbb{R},+\right).
\]
If $\left(1,b\right)\in \mathrm{Aff}(\mathbb{R})$ and
$\left(c,d\right)\in C_{\mathrm{Aff}(\mathbb{R})}\left(a,b\right)$, we get
\[(c,d)=(1,-b)(c,d)(1,b)=(c,b(c-1)+d)\]
 which implies that
\[d=b(c-1)+d \Rightarrow b(c-1)=0 \Rightarrow b=0\;\textrm{or}\;c=1.\]
 If $b=0$, then
\[C_{\mathrm{Aff}\left(\mathbb{R}\right)}(1,0)=\mathrm{Aff}\left(\mathbb{R}\right).\]
 If $b\neq 0$, then
\[C_{\mathrm{Aff}\left(\mathbb{R}\right)}(1,b)=\{(1,d)\mid d\in \left(\mathbb{R},+\right)\} \cong \left(\mathbb{R},+\right).\]

We collect our calculations in the following result.

\begin{lem}
\label{thm:centralizer-Aff}
If $(1,0)\ne (a,b)\in \mathrm{Aff}\left(\mathbb{R}\right)$, then
\(C_{\mathrm{Aff}\left(\mathbb{R}\right)}\left(a,b\right)\cong
\left(\mathbb{R},+\right)\).
\end{lem}

\subsection{Centralizers in \(H\)}\label{subsec:centralizer-h}
%%%%%%%%%%%%%%%%%%%%%%%%%%%%%%%%%%%%%%%%%%%
%%%%%%%%%%%%%%%%%%%%%%%%%%%%%%%%%%%%%%%%%%%
\iffalse
%%%%%%%%%%%%%%%%%%%%%%%%%%%%%%%%%%%%%%%%%%%
%%%%%%%%%%%%%%%%%%%%%%%%%%%%%%%%%%%%%%%%%%%
\altair{Podemos remover. Lembro que a frase foi inserida para dizer que os centralizadores são infinitos. Na tese, ela não existe.}A simple observation to start our study is that, since
\(H\) is torsion-free by Lemma \ref{thm:torsion-free}, then, for any 
non-trivial $z\in H$, \(\langle z\rangle \cong (\mathbb{Z},+)\) and so \(C_{H}\left(z\right)\) must be infinite. \francesco{FRA (possível ideia): We start by noticing that, since \(H\) is torsion-free by Lemma \ref{thm:torsion-free}, then \(C_{H}\left(z\right)\) is infinite for any non-trivial $z \in H$.}
%%%%%%%%%%%%%%%%%%%%%%%%%%%%%%%%%%%%%%%%%%%
%%%%%%%%%%%%%%%%%%%%%%%%%%%%%%%%%%%%%%%%%%%
\fi 
%%%%%%%%%%%%%%%%%%%%%%%%%%%%%%%%%%%%%%%%%%%
%%%%%%%%%%%%%%%%%%%%%%%%%%%%%%%%%%%%%%%%%%%
We start by noticing that, since \(H\) is torsion-free by Lemma \ref{thm:torsion-free}, the
subgroup \(C_{H}\left(z\right)\) is infinite for any non-trivial $z \in H$. 
Next, we will divide the study of centralizers of the elements from \(H\)
into several cases. 

First, we consider \(z\in H\) without breakpoints,
that is, the case where \(z\) is an affine map. Let us consider 
\(z\left(t\right)=a^{2}t+ab\) for all \(t\in \mathbb{R}\). If
$a\neq\pm1$ we have the following result.
\begin{prop}
	Let \(z\in H\) be so that \(z\left(t\right)=a^{2}t+ab\), for all
	\(t\in \mathbb{R}\), with \(a\neq\pm1\). Then
	\(C_{H}\left(z\right)\cong\left(\mathbb{R},+\right)\).
\end{prop}
\begin{proof}
	First, notice that in this case,
	\(z_{-\infty}=z_{+\infty}=(a^{2},ab)\). 
	%Now observe that, for $c>0$, the following map
	%\[g\left(t\right)=ct+\dfrac{ab(c-1)}{a^{2}-1}=
	%(\sqrt{c})^2t+\sqrt{c}\left(\dfrac{ab(c-1)}{\sqrt{c}(a^{2}-1)}\right),\]
	%belongs in \(C_{H}\left(z\right)\).
	%By the map \(\varphi_{z}\) from Lemma \ref{monod-centralizers},
	%\(\varphi_{z}\left(C_{H}\left(z\right)\right)\leq 
	%C_{\mathrm{Aff}\left(\mathbb{R}\right)}\left(a^{2},ab\right)\).
	A direct calculation shows that 
	%Let us consider the following subset of \(H\)
	\[T\coloneqq\left\{f\in H\mid f\left(t\right)=ct+\dfrac{ab(c-1)}{a^{2}-1},\;\forall\;t\in\mathbb{R},\;c>0\right\}\]
	is a subgroup of \(C_{H}\left(z\right)\). 
	Using the map \(\varphi_{z}\) from Lemma \ref{monod-centralizers} we have
	\[\varphi_{z}\left(T\right)=\left\{\left(c,\dfrac{ab(c-1)}{a^{2}-1}\right)\mid c\in \left(\mathbb{R}_{>0},\cdot\right)\right\}=C_{\mathrm{Aff}(\mathbb{R})}(a^{2},ab).\]
	%Then, \(\varphi_{z}\left(T\right)=C_{\mathrm{Aff}(\mathbb{R})}(a^{2},ab)\). 
	From
	\(\varphi_{z}\left(T\right)\leq \varphi_{z}\left(C_{H}\left(z\right)\right)\leq C_{\mathrm{Aff}\left(\mathbb{R}\right)}(a^{2},ab)\),
	we get \(\varphi_{z}\left(C_{H}\left(z\right)\right)=C_{\mathrm{Aff}(\mathbb{R})}(a^{2},ab)\).
	\noindent Since $\varphi_{z}$ is a group monomorphism, we have
	\(C_{H}\left(z\right)\cong C_{\mathrm{Aff}(\mathbb{R})}(a^{2},ab)\).
	Therefore, \(C_{H}\left(z\right)\cong\left(\mathbb{R},+\right)\)
	by Lemma \ref{thm:centralizer-Aff}.
\end{proof}
\indent From the previous result, we have the following.
\begin{cor}
\label{thm:centralizers-affine}
	Let \(y\in H\) be an element such that \(y=g^{-1}zg\), 
	where \(z,g\in H\) and \(z\) is an affine map \(z\left(t\right)=a^{2}t+ab\), with $a^2\ne 1$. Then 
	\(C_{H}(y)\cong \left(\mathbb{R},+\right)\).
\end{cor}

%%%%%%%%%%%%%%%%%%%%%%%%%%%%%%%
%%%%%%%%%%%%%%%%%%%%%%%%%%%%%%%
\iffalse
%%%%%%%%%%%%%%%%%%%%%%%%%%%%%%%
%%%%%%%%%%%%%%%%%%%%%%%%%%%%%%%

{\large \francesco{2020-02-06: the whole next paragraph us \textbf{unchanged}, however it looks confusing. We start talking about affine maps, then we discover something about
their slope (true, but why is it useful right now?) and then we finish saying that
an element centralizing a translation is periodic (true but we should recall why).}}
Before stating the next result, we observe that if an element 
\(z\in H\) does not have breakpoints, then \(z(t)=a^{2}t+ab\) on 
the whole real line. If \(a\neq\pm1\), \(z\) crosses the diagonal 
line. In this case, \(z\notin H^{<}\cup H^{>}\). Given any element
\(z\in H^{<}\) so that around \(-\infty\) we have
\(z\left(t\right)=a_{0}^{2}t+a_{0}b_{0}\) and 
\(z(t)=a_{n}^{2}t+a_{n}b_{n}\) around \(+\infty\), we have 
\(a_{0}^{2}\ge1\) and \(a_{n}^{2}\leq1\). As a consequence,
centralizers of a translation \(z\in H^{<}\) are periodic maps.

%%%%%%%%%%%%%%%%%%%%%%%%%%%%%%%
%%%%%%%%%%%%%%%%%%%%%%%%%%%%%%%
\fi
%%%%%%%%%%%%%%%%%%%%%%%%%%%%%%%
%%%%%%%%%%%%%%%%%%%%%%%%%%%%%%%

Now we consider the case 
\(z\left(t\right)=a^{2}t+ab\) for all \(t\in \mathbb{R}\) and with
$a = \pm1$, that is, $z$ is a translation.

\begin{prop}\label{centralizer-translations}
	If \(z\in H^{<}\) is a translation, then 
	\(C_{H}\left(z\right)\cong \left(\mathbb{R},+\right)\).
\end{prop}
\begin{proof}
Let \(g\in C_{H}\left(z\right)\). Since $g \in H$ we have
	\(g\left(t\right)=a_{0}^{2}t+a_{0}b_{0}\), for some
	\(a_{0},b_{0}\in \mathbb{R}\), \(t\in\left(-\infty,L\right]\)
	and for suitable \(L\in\mathbb{R}\). The map $g$ commutes with the translation
	\(z\left(t\right)=t+k\), for some \(k<0\), so $g$ is periodic of period $|k|$
	and so, by Remark \ref{thm:affine-commute-translation}, we have \(a_{0}^{2}=1\). 
	Hence $g$ is a translation around $-\infty$ and it is periodic, so we
	must have that $g(t)=t+b_0$ for every $t \in \mathbb{R}$.
	Therefore, if $\varphi_z$ is the map of Lemma \ref{monod-centralizers}, we have
	\[
	C_{H}\left(z\right)\cong \varphi_z(C_{H}\left(z\right)) \cong \{(1,b_0) \mid b_0 \in \mathbb{R}\}\cong (\mathbb{R},+).
	\]
\end{proof}
The previous proposition implies the following result.
\begin{cor}\label{cor:centralizers-translations}
	Let \(y\in H\) be an element such that \(y=g^{-1}zg\), where 
	\(z,g\in H\) and \(z\) is a translation. Then 
	\(C_{H}\left(y\right)\cong \left(\mathbb{R},+\right)\).
\end{cor}

Let us now consider \(z\in H^{<}\) such that 
\(z'\left(-\infty\right)\neq z'\left(+\infty\right)\)
and \(z\) has breakpoints. 
%As a consequence, we have \(z_{-\infty}\neq z_{+\infty}\). 
We start with the following result.
\begin{lem}\label{the-goddamn-lemma-5.3-from-Kas-Mat}
	Given \(z\in H^{<}\) such that its initial and final 
	affinity boxes with respect to \(z\) and itself are 
	\(\left(-\infty,L\right]^{2}\) and
	\(\left[R,+\infty\right)^{2}\), respectively, and so that
	\(z'\left(-\infty\right)\neq z'\left(+\infty\right)\). 
	Let
	\(s\in\mathbb{Z}_{>0}\) be such that \(z^{s}\left(R\right)<L\). 
	Then either \(z^{-s}\) is not affine on
	\(\left[z^{s}\left(R\right),L\right]\) or \(z^{-2s}\) is not affine on \(\left[z^{2s}\left(R\right),L\right]\).
\end{lem}

\begin{proof}
	First of all, since $z\in H^{<}$, then $z^{-1}\in H^{>}$.
	Let us suppose that $z\left(t\right)=a_{0}^{2}t+a_{0}b_{0}$ on $(-\infty,L]$
	and $z\left(t\right)=a_{n}^{2}t+a_{n}b_{n}$ on $[R,+\infty)$. 
	Then by hypothesis, $a_{0}^{2}=z'\left(-\infty\right)\neq z'\left(+\infty\right)
	=a_{n}^{2}$.
	Moreover, 
	$z^{-1}\left(t\right)=a_{0}^{-2}t-a_{0}^{-1}b_{0}$ on 
	$(-\infty,z(L)]$ and 
	$z^{-1}\left(t\right)=a_{n}^{-2}t-a_{n}^{-1}b_{n}$ on
	$[z(R),+\infty)$. 
	Then since $z^{-1}\in H^{>}$,
	we have $z^{-s}$ is affine on $(-\infty,z^{s}(L)]$, 
	with initial germ
	\[(z^{-s})_{-\infty}=\left(a_{0}^{-2s},-\sum_{j=1}^{s}a_{0}^{-2j+1}b_{0}\right).\]
	\noindent Moreover, $z^{-s}$ is affine on
	$[z^{s}(R),+\infty)$, which contains $[R,+\infty)$
%	\red{FRA: I don't understand this last sentence. A few lines before you wrote that
%	$z^{-s}$ is affine on $(-\infty,z^{s}(L)]$ (which is correct), shouldn't you now write
%	``$z^{-s}$ is affine on $[z^{s}(R),+\infty) \supseteq [R,+\infty)$''? You actually do use that 
%	$z^{-s}$ is affine on $[z^{s}(R),+\infty)$ a few lines below.}
%	\green{ALT: My mistake. Sorry.}
	with final germ
	\[(z^{-s})_{+\infty}=\left(a_{n}^{-2s},-\sum_{j=1}^{s}a_{n}^{-2j+1}b_{n}\right).\]
	Let us assume, by contradiction, that both $z^{-s}$ and 
	$z^{-2s}$ are both affine on $[z^{s}(R),L]$ and 
	$[z^{2s}(R),L]$, respectively, and that their 
	germs on these two intervals are $(a,b)$ and $(c,d)$, respectively.
	Since $z^{-2s}=z^{-s}\circ z^{-s}$, we get
	$z^{-2s}$ is affine on $[z^{s}(R),L]$,
	because $z^{-s}$ is affine on $[z^{s}(R),L]$
	by our assumption and $z^{-s}$ is affine on 
	$[R,z^{-s}(L)]\subset [R,+\infty)$, 
	with germ
	\[\left(a_{n}^{-2s},-\sum_{j=1}^{s}a_{n}^{-2j+1}b_{n}\right)(a,b).\]
	\noindent Moreover, $z^{-2s}$ is also affine on
	$[z^{2s}(R),z^{s}(L)]$, since $z^{-s}$ is affine on 
	$(-\infty,z^{s}(L)]$
	and on $[z^{s}(R),L]$ by our assumption, 
	with germ
	\[(a,b)\left(a_{0}^{-2s},-\sum_{j=1}^{s}a_{0}^{-2j+1}b_{0}\right).\]
	\noindent By comparing the germ $(c,d)$ of $z^{-2s}$ on $[z^{2s}(R),L]$ 
	with the germs of the same map $z^{-2s}$ on the two subintervals
	$[z^{s}(R),L]$ and $[z^{2s}(R),z^{s}(L)]$ of the interval $[z^{2s}(R),L]$,
	we get
	\[\left(a_{n}^{-2s},-\sum_{j=1}^{s}a_{n}^{-2j+1}b_{n}\right)(a,b)=(c,d)=(a,b)\left(a_{0}^{-2s},-\sum_{j=1}^{s}a_{0}^{-2j+1}b_{0}\right).\]
	%In this way,
	%\[\left(a_{n}^{-2s},-\sum_{j=1}^{s}a_{n}^{-2j+1}b_{n}\right)(a,b)=(a,b)\left(a_{0}^{-2	s},-\sum_{j=1}^{s}a_{0}^{-2j+1}b_{0}\right).\]
	From this, we must have
	\[a_{n}^{-2s}a=aa_{0}^{-2s}.\]
	Since the group $(\mathbb{R}_{>0},\cdot)$
	is abelian, we have $a_{0}^{-2s}=a_{n}^{-2s}$.
	However, we are considering $z\in H^{<}$
	such that the $a_{0}^2\neq a_{n}^2$, so that $a_{0}^{-2s} \ne a_{n}^{-2s}$
%	{\Huge \red{FRA: when did you assume that $a_0 \ne a_n$? What is wrong with this? I think that here you are missing hypotheses, are you not? This is where we were using that argument from page 70 of the current version of the PDF (between Corollary 5.2 and Proposition 5.3). I mean that from that argument we got $a_0^2 \ge 1$ and $a_n^2 \le 1$ and, if either 
%	$a_0^2>1$ or $a_n^<1$ we get a contradiction. In the case $a_0^2=a_n^2=1$ we have no contradiction and that is the bad case, the one where sometimes the centralizer is $\mathbb{R}$ and sometimes it is $\mathbb{Z}$. I think that the hypotheses must contain that either
%	$z'(-\infty)>1$ or $z'(-\infty)<1$, which correspond to $a_0^2>1$ or $a_n^<1$.}}
	%\green{ALT: I commented your red text here. I think I fixed the problem.}
	and we have a contradiction. Therefore, either
	$z^{-s}$ is not affine on $[z^{s}(R),L]$ or
	$z^{-2s}$ is not affine on $[z^{2s}(R),L]$. 
\end{proof}

\begin{lem}\label{thm:unbounded-to-bounded}
	The group $H$ is isomorphic to the group $K$ of all piecewise M\"obius transformations of $[0,1]$ to itself with finitely many breakpoints.
	%Let $K$ be the group of all piecewise M\"obius transformations of $[0,1]$ to itself
	%with finitely many breakpoints. Then, there is an isomorphism $\nabla:H \to K$.
	
	%\francesco{Talvez possamos Chamar o grupo do enunciado de $K$ e, no começo da prova, falamos do que $K \cong H|_{[0,1]}$, o grupo das restrições dos elementos de $H$?} \altair{Posso estar fazendo confusão, mas não vejo o novo grupo como se simplemente pegasse os elementos de \(H\) e restringisse o domínio ao intervalo, pois é assim que parece-me estar fazendo quando escreve $\{h \in H \mid h(t)=t, \forall\; t \not \in (0,1) \}$. O que eu vejo que fazemos é que "jogamos" \(\mathbb{R}\) para o intervalo \([0,1]\), usando a conjugação feita por \(\nabla\). Podemos conversar melhor sobre isso em uma reunião rápida no \textit{Meet}}
\end{lem}

\begin{proof}%[Idea of the proof]
%By definition, the group $K$ acts on the interval $[0,1]$, but it is clearly isomorphic
%to the subgroup $H|_{[0,1]}$ of $H$ given by the homeomorphisms supported on $[0,1]$.

We explicitly construct an isomorphism
%$\nabla:H \to H|_{[0,1]}$.
$\nabla:H \to K$.
%%%%%%%%%%%%%
\iffalse
%%%%%%%%%%%%%

%$$
%K \cong H|_{[0,1]} = \{h \in H \mid h(t)=t, \forall\; t \not \in (0,1) \}.
%$$

Notice that $H|_{[0,1]}$ is the group obtained by taking $\{h \in H \mid h(t)=t, \forall\; t \not \in (0,1) \}$ \altair{Creio que esse conjunto não descreve o \(H|_{[0,1]}\) propriamente, pois não é obrigatório ter somente o intervalo \([0,1]\) como suporte.}
\francesco{Acho que não entendo o que queira dizer. Desculpe. Talvez quer dizer que a primeira frase deste paragrafo teria de ser ``Notice that $H|_{[0,1]}$ is \textbf{isomorphic to the} group obtained by taking''? Neste caso, concordaria pois a descrição de $H|_{[0,1]}$ feita no enunciado. Verdadeiramente, se este é o que quiser dizer (e me parece fazer super sentido), é melhor usar o simbolo $H_{[0,1]}$ em lugar de $H|_{[0,1]}$ (no primeiro caso não falamos de que seja uma restrição de $H$)} and restricting its elements to $[0,1]$.
	
	Now we construct the isomorphism explicitly. 

%%%%%%%%%%%%%%%%%%
\fi
%%%%%%%%%%%%%%%%%%
We use Lemma \ref{monod-transitive} to construct
	an element $h \in H$ such that $h(-3)=\frac{1}{3}$ and $h(-1)=\frac{1}{2}$. Now
	consider the map
	\[
	f\left(t\right)=\begin{dcases}
	\frac{-1}{t} & t \in \left(-\infty,-3\right] \\
	h\left(t\right) & t \in [-3,-1] \\
	\frac{t+2}{t+3} & t \in \left[-1,+\infty\right).
	\end{dcases}
	\]
%%%%%%%%%%%%%
\iffalse
%%%%%%%%%%%%%	
	
	We now define
	$\nabla(g)=fgf^{-1}$ and notice that a direct calculation shows 
	$\mathrm{im}(\nabla) \subseteq K$. Thus the map
	%$\mathrm{im}(\nabla) \subseteq H|_{[0,1]}$. Thus the map 
	$\nabla:H \to K$ is well-defined and it is
	%$\nabla:H \to H|_{[0,1]}$ is well-defined and it is 
	clearly a group isomorphism.
%%%%%%%%%%%%%
\fi
%%%%%%%%%%%%%
We now define
	\[
	\nabla(g)(t)=\begin{dcases}
	fgf^{-1}(t) & t \in (0,1) \\
	t & t \in \{0,1\}
	\end{dcases}
	\]
	and notice that a direct calculation shows 
	$\mathrm{im}(\nabla) \subseteq K$. Thus the map
	%$\mathrm{im}(\nabla) \subseteq H|_{[0,1]}$. Thus the map 
	$\nabla:H \to K$ is well-defined and it is
	%$\nabla:H \to H|_{[0,1]}$ is well-defined and it is 
	clearly a group isomorphism with an obvious inverse.
\end{proof}

\begin{rem}\label{mobius-boxes}
    The isomorphism $\nabla$ of the proof of Lemma~\ref{thm:unbounded-to-bounded} switches $-\infty$ with $0$ and $+\infty$ with $1$ and
	allows us to study maps in Monod's group from a bounded point of view which will be useful in the proof of Lemma \ref{thm:generalized-KasMat-5.3}. Moreover, a straightforward calculation shows that, if $y,z \in H$ are such that $y_{-\infty}=z_{-\infty}$ and $y_{+\infty}=z_{+\infty}$,
	then the initial and final affinity boxes of $y$ and $z$ correspond to
	initial and final \emph{M\"obius boxes} of $\nabla(y)$ and $\nabla(z)$, where 
	the images coincide and are M\"obius and a conjugator has to be M\"obius.
\end{rem}

In the next result we will freely use the isomorphism $\nabla:H \to K$
of Lemma~\ref{thm:unbounded-to-bounded}.

\begin{lem}
	\label{thm:generalized-KasMat-5.3}
	Let \(z\in H^{<}\) be such that \(z\left(t\right)=a^2t+ab\) at 
	\(-\infty\) with \(a^2>1\). 
	%If $\nabla$ is the isomorphism of Lemma~\ref{thm:unbounded-to-bounded},
	Then there exists \(\varepsilon>0\) 
	such that the only \(g\in C_{H}\left(z\right)\) with	
	\(1-\varepsilon<\widetilde{g}~'(0)<1+\varepsilon\) and 
	\(-\varepsilon<\widetilde{g}~''(0)<\varepsilon\), where $\widetilde{g}=\nabla(g)$, is
	\(g=\mathrm{id}\).
\end{lem}

\begin{proof}
	Let us consider $\widetilde{z}$ to be conjugate version of $z$ from the proof of Lemma~\ref{thm:unbounded-to-bounded}, that is, $\widetilde{z}=\nabla(z)$.
	%\[\widetilde{z}\left(t\right)=\dfrac{t}{-abt+a^{2}},\]
	%\noindent the conjugated version of $z$ by $f$ defined 
	%above. 
	Let $[0,\alpha]$ and $[\beta,1]$ be, respectively, the
	initial and final M\"obius boxes of $\widetilde{z}$ (see Remark \ref{mobius-boxes})
	for suitable $0<\alpha<\beta<1$.
	By Lemma \ref{the-goddamn-lemma-5.3-from-Kas-Mat}, there exists an $N_1\in\mathbb{Z}_{>0}$
	such that $\widetilde{z}^{-N_1}$ has a breakpoint $\mu_1$ on $[\widetilde{z}^{N_1}\left(\beta\right),\alpha]$. We now consider a real number
	$\alpha'$ such that
	$0<\alpha'<\mu_1<\alpha$ and we take a new initial (and smaller) M\"obius box $[0,\alpha']$ for $z$,
	we use Lemma \ref{the-goddamn-lemma-5.3-from-Kas-Mat} again and find that there exists
	$N_2\in\mathbb{Z}_{>0}$
	such that $\widetilde{z}^{-N_2}$ has a breakpoint $\mu_2$ on $[\widetilde{z}^{N_2}\left(\beta\right),\alpha']$. Without loss of generality, assume that $\widetilde{z}^{N_2}\left(\beta\right)\le
	\widetilde{z}^{N_1}\left(\beta\right)$.
	Then there exists $\varepsilon>0$ such that 
	$\{\mu_2<\mu_1 \} \subseteq I_{\varepsilon}\coloneqq \left[\widetilde{z}^{N_2}\left(\dfrac{\beta+\varepsilon}{1+\varepsilon}\right),(1-\varepsilon)\alpha\right]$.
	
	\begin{claim}\label{ze-plane-ze-plane}
		Let $0 < \varepsilon <\frac{1}{3}$ and $g\in C_{H}\left(z\right)$ such that
		$$1-\varepsilon<\widetilde{g}~'(0)<1+\varepsilon\;\;\textrm{and}\;\; -\varepsilon<\widetilde{g}~''(0)<\varepsilon.$$
		Then $|\widetilde{g}(t)-\mathrm{id}(t)|<3\varepsilon+2\varepsilon^2$, for all $t \in [0,\alpha]$, so
		the family of functions $\widetilde{g}$ can be seen as uniformly converging to the identity function $\mathrm{id}$ on the interval $[0,\alpha]$.
	\end{claim}
	\begin{proof}[Proof of Claim \ref{ze-plane-ze-plane}]
%		The following is essentially a proof of Lemma \ref{cauchyproblem} suitably adapted to prove the claim.
		Let us consider \(\widetilde{g}=\nabla\left(g\right)\) so 
		that \(\widetilde{g}\left(t\right)=\dfrac{at+b}{ct+d}\) on 
		\(\left[0,\alpha\right]\), where \(ad-bc=1\). Then
		\(\widetilde{g}\left(0\right)=0\) and, consequently, \(b=0\) and \(ad=1\). 
		Let us define 
		\(\widetilde{g}'\left(0\right)\coloneqq\lambda\) and 
		\(\widetilde{g}''\left(0\right)=\rho\). Since
		\[\widetilde{g}'\left(t\right)=\dfrac{1}{\left(ct+d\right)^{2}}\;\;\textrm{and}\;\;\widetilde{g}''\left(t\right)=-\dfrac{2c}{\left(ct+d\right)^{3}},\]	
		we have \(\lambda=\dfrac{1}{d^2}\) and \(\rho=-\dfrac{2c}{d^3}\). Therefore, 
		\(d^2=\dfrac{1}{\lambda}\) and \(c=\dfrac{-\rho d^3}{2}\). Observe that
		\[
		\widetilde{g}\left(t\right)=\dfrac{at}{ct+d}=\dfrac{t}{cdt+d^2}=\dfrac{t}{\frac{-\rho t}{2 \lambda^2}+\frac{1}{\lambda}}=
		\dfrac{2\lambda^2 t}{-\rho t + 2\lambda}
		\]
		and so
		\begin{align*}
			|\widetilde{g}\left(t\right)-\mathrm{id}\left(t\right)|&=\left|
			\dfrac{2\lambda^2 t}{-\rho t + 2\lambda}-t
			\right|\\
			&=\left|\frac{2\lambda^2 t - 2\lambda t+\rho t^2}{-\rho t + 2\lambda}\right|\\
			&\le \left|2\lambda^2 t - 2\lambda t+\rho t^2\right|\\
			&\le 2\left|\lambda\right|\cdot\left|\lambda-1\right|\cdot\left|t\right|+\left|\rho\right|\cdot\left|t\right|\\
			&\le 2\left(1+\varepsilon\right)\varepsilon+\varepsilon\\
			&\le 3\varepsilon + 2\varepsilon^2
		\end{align*}
		where at the various steps we have observed that \(\left|t\right|\le1\),
		\(\left|\lambda\right|\le1+\varepsilon\),
		\(\left|\lambda -1\right|\le\varepsilon\) and, since 
		$|\rho|<\varepsilon<\frac{1}{3}$, we have
		\[
		|-\rho t+2\lambda | \ge |2\lambda-|\rho t|| \ge |2\lambda - \varepsilon|
		\ge |2(1-\varepsilon) - \varepsilon|=|2 - 3\varepsilon| \ge 1.
		\]
	\end{proof}
	\begin{claim}\label{again-ze-plane}
        Let $t_0 \in (0,\alpha)$. Then for any
        $1-\varepsilon<\widetilde{g}'(0)=\lambda <1+\varepsilon$,
		there is at most one $g \in C_H(z)$ such that $-\varepsilon<\widetilde{g}''(0)=\rho<\varepsilon$ and such that
		$\widetilde{g}^{-1}(t_0)=t_0$.
	\end{claim}

	\begin{proof}[Proof of Claim \ref{again-ze-plane}]
		We write $\widetilde{g}$ on the open interval $(0,\alpha)$ using the expression that was computed in the proof of Claim \ref{ze-plane-ze-plane}. Assume that $\widetilde{g}(t_0)=t_0$, then
		\[
		t_0=\frac{2\lambda^2t_0}{-\rho t_0 + 2\lambda}
		\]
		and so
		\[
		1=\frac{2\lambda^2}{-\rho t_0 + 2\lambda}
		\]
		and so
		\[
		-\rho t_0 + 2\lambda=2\lambda^2
		\]
		and so
		\[
		\rho=\frac{2\lambda-2\lambda^2}{t_0}
		\]
		If we assume that $\lambda=1+\tau$ for $-\varepsilon \le \tau \le \varepsilon$, then
		\[
		\rho=\frac{2(1+\tau)-2(1+\tau)^2}{t_0}=\frac{-2\tau-2\tau^2}{t_0}.
		\]
		For any $-\varepsilon \le \tau \le \varepsilon$, the expression above returns a unique
		$\rho$. In case such expression returns $|\rho| \ge \varepsilon$, then $g$ cannot exist. On the other hand, if such expression returns $|\rho|< \varepsilon$, then the pair $(\tau,\rho)$ satisfies the required conditions. Therefore, for each $\lambda$ and we obtain at most one $g$
		satisying the requirements.
	\end{proof}

	\noindent \emph{End of the Proof of Lemma \ref{thm:generalized-KasMat-5.3}.}
	Since we know that
	\begin{itemize}
	    \item[(i)] \(\mu_{i}\) is a breakpoint for \(\widetilde{z}^{-N_{i}}\),
	    \item[(ii)] \(\widetilde{z}^{-N_{i}}\left(\mu_{i}\right)\in\left[0,\alpha\right]\), and
	    \item[(iii)] \(\widetilde{g}\) is a M\"obius transformation on \(\left[0,\alpha\right]\),
	\end{itemize}
	
    \noindent it follows that \(\mu_{i}\) is a
	breakpoint for \(\widetilde{g}\widetilde{z}^{-N_{i}}\).
	On the other hand, when we consider 
	\(\widetilde{z}^{-N_i}\widetilde{g}\), the map \(\widetilde{g}\) 
	pushes the breakpoint \(\mu_i\) of \(\widetilde{z}^{-N_i}\) to 
	\(\widetilde{g}^{-1}(\mu_i)\), then
	\(\widetilde{g}^{-1}\left(\mu_{i}\right)\) is a breakpoint for
	\(\widetilde{z}^{-N_i}\widetilde{g}\).

	By construction, the set of breakpoints of $\widetilde{g}\widetilde{z}^{N_i}$ on $I_\varepsilon$
is $B:=\{\delta_1< \ldots < \delta_k \} \supseteq \{\mu_1<\mu_2\}$
and the set of breakpoints of $\widetilde{z}^{N_i}\widetilde{g}$ on $\widetilde{g}^{-1}(I_\varepsilon)$ is
$\widetilde{g}^{-1}(B)=\{\widetilde{g}^{-1}(\delta_1)< \ldots < \widetilde{g}^{-1}(\delta_k) \}
\{\widetilde{g}^{-1}(\mu_1)<\widetilde{g}^{-1}(\mu_2)\}$.
However, since \(g\in C_{H}\left(z\right)\), 
then 
	\(\widetilde{g}\widetilde{z}^{N_i}\left(t\right)=\widetilde{z}^{N_i}\widetilde{g}\left(t\right)\), 
	for every \(t\in I_{\varepsilon}\) and so
	\(\widetilde{g}^{-1}\left(\delta_i\right)=\delta_i\) for \(i=1,\ldots,k\) and
	in particular \(\widetilde{g}^{-1}\left(\mu_i\right)=\mu_i\) for \(i=1,2\).
	
	By Claim \ref{again-ze-plane}, there can exist at most one $\mathrm{id}\ne g \in C_H(z)$ 
	fixing $\mu_1$ and, since $\widetilde{g}$ fixes $0$ too, it cannot also
	fix $\mu_2$, otherwise \(g\) would be the identity map, by
	\cite[Corollary 2.5.3]{JonSin1987}. Similarly, there can exist at most one
	$\mathrm{id}\ne g \in C_H(z)$ fixing $\mu_2$ and such map cannot fix $\mu_1$ too. 
	Then the only way to avoid a contradiction and have a \(g \in 
	C_{H}\left(z\right)\)
	such that \(\widetilde{g}'\left(0\right)\) and
	\(\widetilde{g}''\left(0\right)\) satisfy
	the given conditions with respect to the chosen \(\varepsilon>0\) 
	is that \(g=\mathrm{id}\).
\end{proof}

We now show that in many cases centralizers are infinite cyclic.

\begin{prop}\label{monod-centralizer-breakpoints-affine}
	Let \(z\in H^{<}\) be such that \(z\left(t\right)=a^2t+ab\) around 
	\(-\infty\) and \(a^2>1\). Then \(C_{H}\left(z\right)\) is a
	discrete subgroup of \(\left(\mathbb{R},+\right)\) and so it is 
	isomorphic to \((\mathbb{Z},+)\).
\end{prop}
\begin{proof}
	By Lemma \ref{thm:generalized-KasMat-5.3}, the subgroup
	\(C_{H}\left(z\right)\) is a discrete set. Since 
	\(C_{H}\left(z\right)\cong\varphi_z\left(C_{H}\left(z\right)\right)\le C_{\mathrm{Aff}\left(\mathbb{R}\right)}\left(z\right)\cong \left(\mathbb{R},+\right)\)
	and the subgroups of \(\left(\mathbb{R},+\right)\) are either 
	discrete (then isomorphic to \(\left(\mathbb{Z},+\right)\)), 
	or dense we get
	\(C_{H}\left(z\right)\cong\left(\mathbb{Z},+\right)\).
\end{proof}
\subsubsection{Mather invariant and centralizers}
\label{subsec:mather-centralizer}
	As done is Section \ref{sec:mather}, we consider 
	\(z\in H^{>}\) that is a translation around $\pm \infty$ and
	%so that \(z\left(t\right)=t+b_{0}\) for
	%\(t\in\left(-\infty,L\right]\), \(z\left(t\right)=t+b_{1}\) for
	%\(t\in\left[R,+\infty\right)\) and 
	%\(z^{N}\left(\left(z^{-1}\left(L\right),L\right)\right)\subset\left(R,+\infty\right)\) 
	%for some \(N\in\mathbb{Z}_{>0}\) sufficiently large. 
	we use
	the Mather invariant of \(z\) in order to understand centralizers.
	
	\begin{prop}\label{thm:Centralizer-Last-Case}
		Consider \(z\in H^{>}\) such that 
		\(z\left(t\right)=t+b_{0}\) for \(t\in\left(-\infty,L\right]\) and
		\(z\left(t\right)=t+b_{1}\) for \(t\in\left[R,+\infty\right)\).
		%and $N \in \mathbb{Z}_{>0}$ large enough so that
		%\(z^{N}\left(\left(z^{-1}\left(L\right),L\right)\right)\subset\left(R,+\infty\right)\). 
		Then either
		\(C_{H}\left(z\right)\cong\left(\mathbb{Z},+\right)\) or 
		\(C_{H}\left(z\right)\cong\left(\mathbb{R},+\right)\).
	\end{prop}
	\begin{proof}
		We follow notations from Section \ref{sec:mather}. Let 
		$N \in \mathbb{Z}_{>0}$ large enough so that
		\(z^{N}\left(\left(z^{-1}\left(L\right),L\right)\right)\subset\left(R,+\infty\right)\). 
		Up to conjugating \(z\) with \(s\), we will work with 
		\(z\left(t\right)=t+1\). We define the relation	\(t\sim t+1\) 
		and construct the circles 
		\(C_{0}\coloneqq\left(-\infty,0\right]/\sim\) and 
		\(C_{1}\coloneqq\left[N-1,+\infty\right)/\sim\).
		By Theorem \ref{matconj}, a \(g\in H\) is a centralizer of \(z\)
		if and only if the following equation is satisfied 
		\begin{align}\label{cent}
		z^{\infty}v_{0,\ell}=v_{1,m}z^{\infty}.
		\end{align}
		We now consider the map $V_{0}\colon\mathbb{R}\longrightarrow\mathbb{R}$
		defined by \(V_0(t)=t+\ell\), which is a lift of of $v_{0,\ell}$, that is, 
		it makes the the following diagram commute
		\[\xymatrix{\ar @{} [dr] |{\circlearrowright}
			\mathbb{R} \ar[d]_-{p_{0}} \ar[r]^-{V_{0}} & \mathbb{R} \ar[d]^-{p_{0}} \\
			C_{0} \ar[r]_{v_{0,\ell}} & C_{0}
		}\]
		Similarly \(V_{1}(t)=t+m\) makes the following diagram commute
		\[\xymatrix{\ar @{} [dr] |{\circlearrowright}
			\mathbb{R} \ar[d]_-{p_{1}} \ar[r]^-{V_{1}} & \mathbb{R} \ar[d]^-{p_{1}} \\
			C_{1} \ar[r]_{v_{1,m}} & C_{1}
		}\]
		Let $Z:\mathbb{R}\to\mathbb{R}$ be a lift of \(z^{\infty}\). 
		The previous two commutative diagrams and equation (\ref{cent}) 
		form three faces of a commutative cube analogous to that appearing in the proof of Theorem \ref{matconj} and so they imply that
		\(ZV_{0}=V_{1}Z\). In other words, for \(t\in\mathbb{R}\), we have
		%Given \(t\in\mathbb{R}\), we get
		%\(ZV_{0}\left(t\right)=Z(t+\ell)\) and
		%\(V_{1}Z\left(t\right)=Z\left(t\right)+m\).
		\begin{align}\label{cent-lifting}
		Z(t+\ell)=ZV_{0}\left(t\right)=V_{1}Z\left(t\right)=Z\left(t\right)+m,
		\end{align}
		which means that the graph of \(Z\)	is shifted back to itself.
		If the lift of \(z^{\infty}\) does not have 
		breakpoints, the graph of \(Z\) is affine. Thus, there are 
		infinitely many pairs \(\ell,m\in\mathbb{R}\) for which the 
		graph can be shifted back to itself and so, for each 
		\(\ell\in\mathbb{R}\), there exists an \(m\in\mathbb{R}\)
		so that equation (\ref{cent-lifting}) holds. Consequently, the 
		image of the map \(\varphi_{z}\) from Lemma
		\ref{monod-centralizers} is so that
		\(\varphi_{z}\left(C_{H}\left(z\right)\right)\cong\left(\mathbb{R},+\right)\).
		Otherwise, the lift of \(z^{\infty}\) has 
		breakpoints and the set of candidates for \(\ell\) forms 
		a discrete subgroup of \(\left(\mathbb{R},+\right)\). Then
		\(\varphi_{z}\left(C_{H}\left(z\right)\right)\cong \left(\mathbb{Z},+\right)\).
		Therefore we have either	
		\(C_{H}\left(z\right)\cong\left(\mathbb{Z},+\right)\)
		or \(C_{H}\left(z\right)\cong\left(\mathbb{R},+\right)\).
	\end{proof}
	We see two examples: in the first one, the subgroup of centralizers is isomorphic 
	to $\left(\mathbb{R},+\right)$, while in the second it is isomorphic to
	$\left(\mathbb{Z},+\right)$.
	\begin{ex}\label{Ex-Translation-Conjugated}
		If we conjugate \(y\left(t\right)=t+1\) by 
		\[g\left(t\right)=\begin{dcases}
		\dfrac{t-2}{\frac{3}{2}t-2},&\textrm{if}\;t\in[0,1],\\
		t+1,&\textrm{otherwise}.
		\end{dcases}
		\]
		\noindent we get
		\begin{equation*}
			z\left(t\right)=\begin{dcases}
			\dfrac{2t+2}{\frac{3}{2}t+2},&\textrm{if}\;t\in[-1, 0],\\
			\dfrac{t-2}{\frac{3}{2}t-2},&\textrm{if}\;t\in[0,1],\\
			t+1,&\textrm{otherwise.}
			\end{dcases}
		\end{equation*}
		Then \(C_{H}\left(z\right)\cong \left(\mathbb{R},+\right)\), by Corollary \ref{cor:centralizers-translations}.
		\begin{figure}[!ht]
			\centering
			\includegraphics[scale=.4]{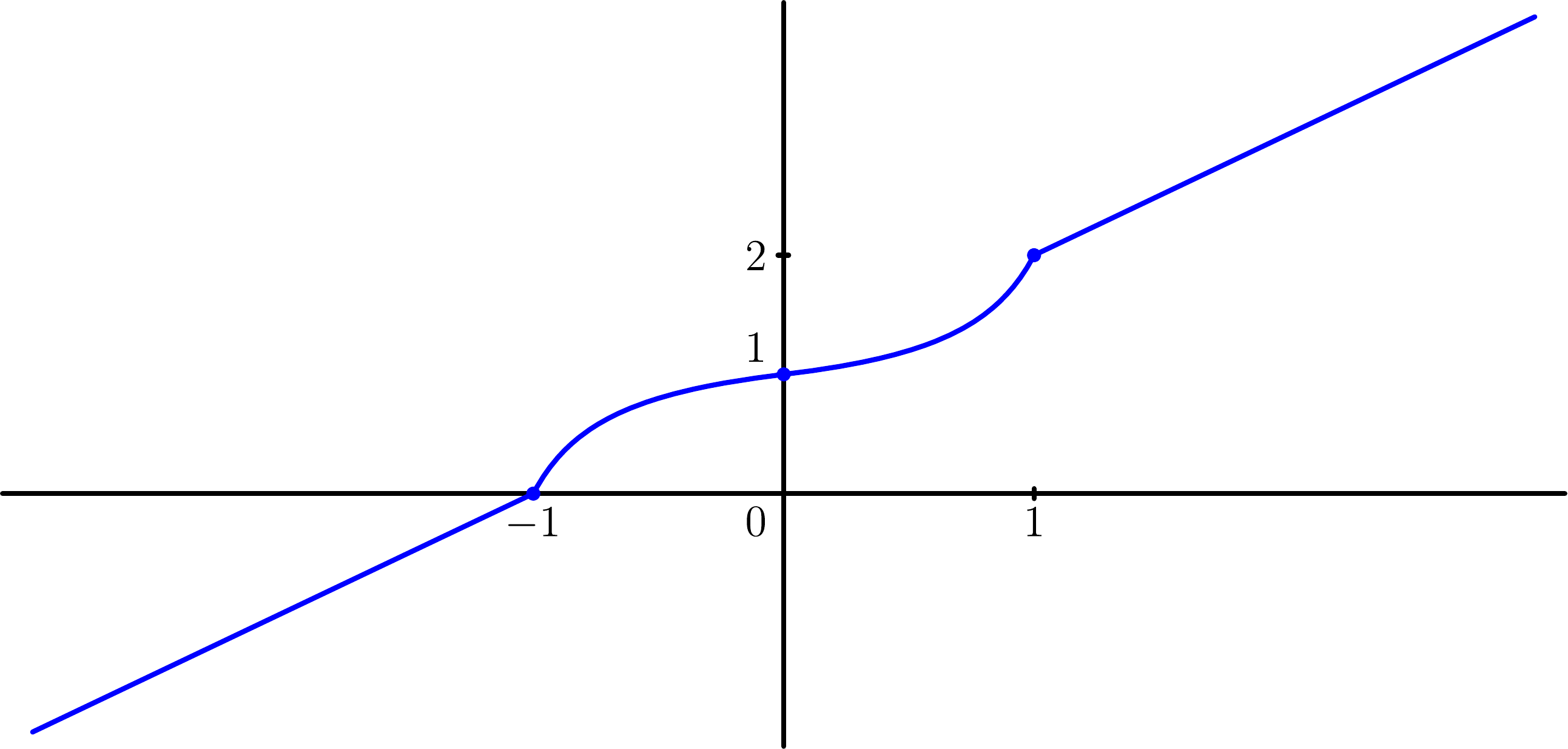}
			\caption{Graph of \(z\), from Example \ref{Ex-Translation-Conjugated}.}
			\label{GraphEx5}
		\end{figure}
	\end{ex}
	\begin{ex}\label{Ex-Discrete}
		Let us consider
			\[z\left(t\right)=\begin{dcases}
			\dfrac{t-2}{\frac{3}{2}t-2},&\;\textrm{if}\;t\in[0,1];\\
			t+1,&\;\textrm{otherwise}.
			\end{dcases}\]
		Notice that \(z\in H^{>}\) and that \(L=0\) and \(R=1\). 
		\begin{figure}[!ht]
			\centering
			\includegraphics[scale=.4]{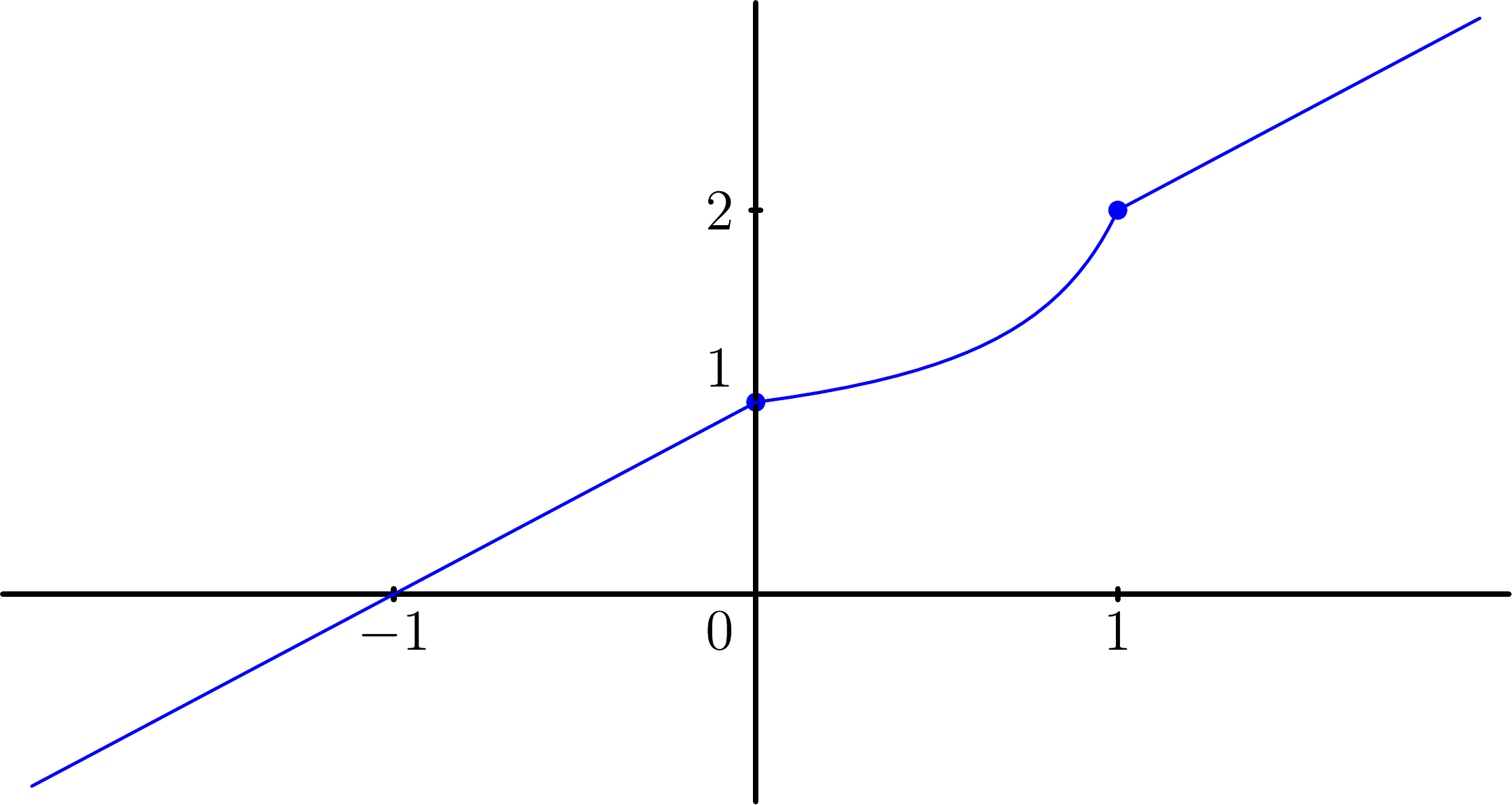}
			\caption{Graph of \(z\), from Example \ref{Ex-Discrete}.}
		\end{figure}
		\noindent Its inverse is given by
			\[z^{-1}\left(t\right)=\begin{dcases}
			\dfrac{2t-2}{\frac{3}{2}t-1},&\;\textrm{if}\;t\in[1,2];\\
			t-1,&\;\textrm{otherwise}.
			\end{dcases}\]	
		If \(N=2\), we have
		\[z^{2}((z^{-1}(0),0))\subset \left[1,+\infty\right).\]
		Notice that we do not need to conjugate by the map \(s\), since 
		it is a translation by one around $\pm \infty$, that is \(\overline{z}\coloneqq z\). 
		Moreover,
			\[z^{2}\left(t\right)=\begin{dcases}
			\dfrac{t-1}{\frac{3}{2}t-\frac{1}{2}},&\;\textrm{if}\;t\in[-1,0];\\
			\dfrac{\frac{5}{2}t-4}{\frac{3}{2}t-2},&\;\textrm{if}\;t\in[0,1];\\
			t+2,&\;\textrm{otherwise}.
			\end{dcases}\]
		\begin{figure}[!ht]
			\centering
			\includegraphics[scale=.4]{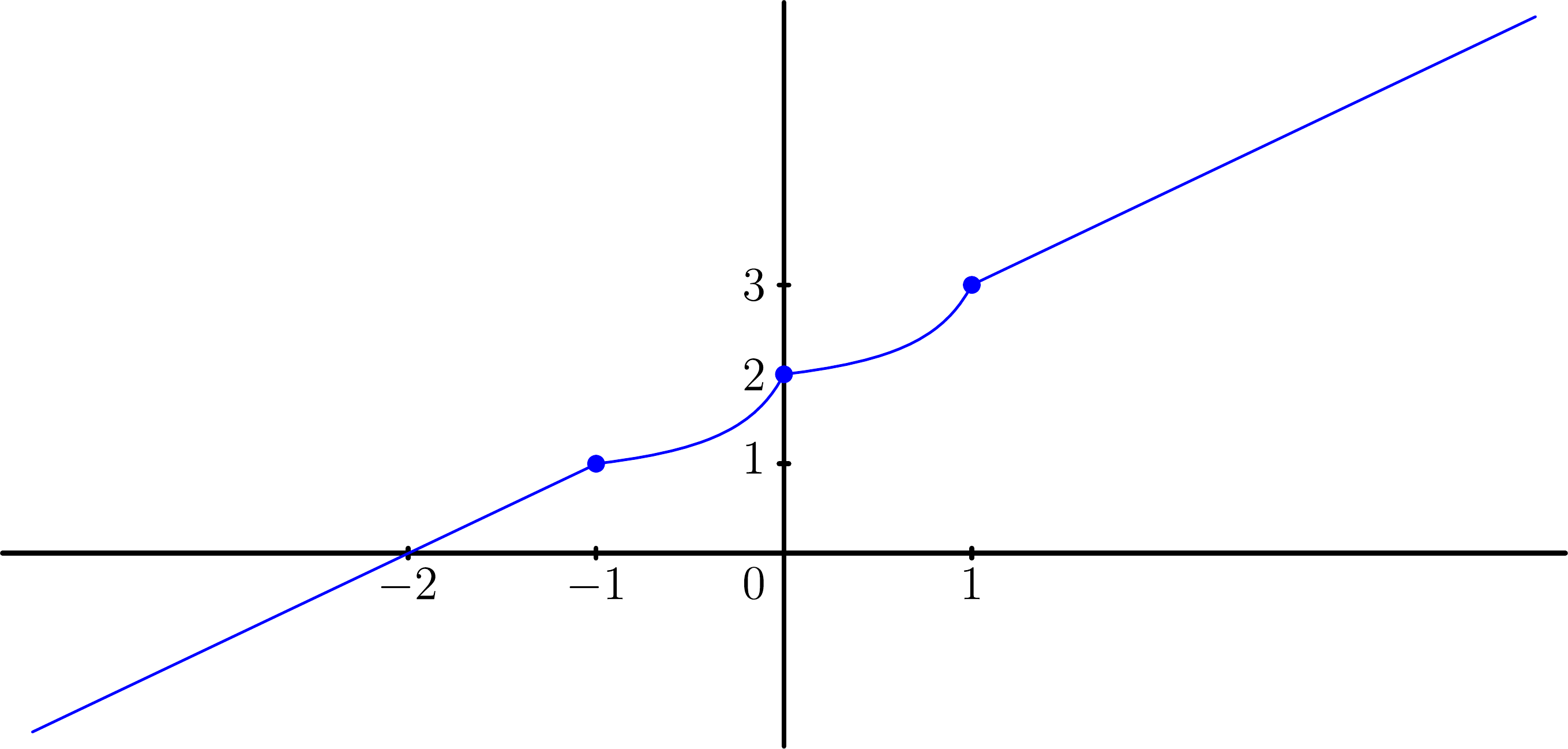}
			\caption{Graph of \(z^{2}\).}
		\end{figure}
		Considering the relation \(t\sim t+1\), we define
		\(C_{0}\coloneqq\left(-\infty,0\right]/t\sim t+1\) and \(C_{1}\coloneqq \left[1,+\infty\right)/t\sim t+1\).
		\noindent Then we get the Mather invariant
		\begin{align*}
			z^{\infty}\colon C_{0} &\longrightarrow C_{1} \\
			\left[t\right] &\longmapsto z^{\infty}\left(\left[t\right]\right)=\left[z^{2}\left(t\right)\right].
		\end{align*}
		\noindent The lift of this map making the following diagram commute
		\[\xymatrix{\ar @{} [dr] |{\circlearrowright}
		\mathbb{R} \ar[d]_-{p_{0}} \ar[r]^-{Z} & \mathbb{R} \ar[d]^-{p_{1}} \\
		C_{0} \ar[r]_-{z^{\infty}} & C_{1}}\]
		\noindent is given by the periodic extension of the restriction of \(z^{2}\) 
		defined on \(\left[-1,0\right]\) by
		\[Z\left(t\right) = z^{2}(t-x)+x,\]
		\noindent if \(x-1\leq t\leq x\), where \(x\in\mathbb{Z}\). Then the 
		centralizer of \(Z\) is \(\left(\mathbb{Z},+\right)\). Moreover, notice that 
		\(Z\notin H\).
		\begin{figure}[!ht]
			\centering
			\includegraphics[scale=.4]{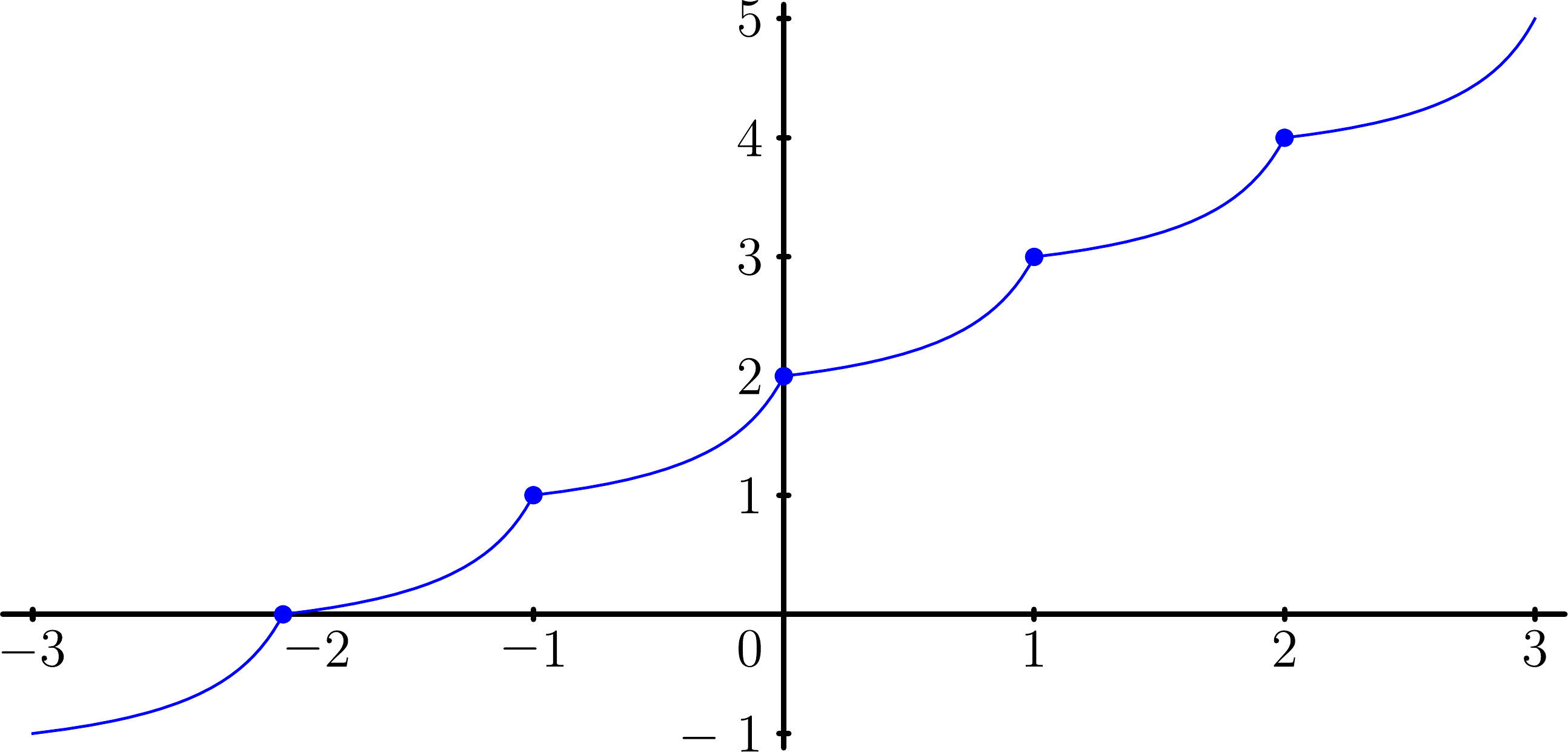}
			\caption{Graph of the lift \(Z\).}
		\end{figure}
	\end{ex}

\subsubsection{Main result about centralizers} We can now give a structure result
for centralizers in $H$ (Theorem A in the introduction).
\begin{thm}
	Given \(z\in H\), then
	\[C_{H}\left(z\right)\cong\left(\mathbb{Z},+\right)^{n}\times\left(\mathbb{R},+\right)^{m}\times H^{k},\]
	for suitable \(k,m,n\in\mathbb{Z}_{\geq 0}\).
\end{thm}
\begin{proof}
The element $z$ has finitely many (possibly unbounded) intervals of fixed points, so its boundary
	\(\partial\mathrm{Fix}\left(z\right)=\left\{t_{0}<t_{1}<\ldots<t_{n}\right\}\)
	has only finitely points.
	If \(g\in C_{H}\left(z\right)\), then \(g\) fixes 
	\(\partial\mathrm{Fix}\left(z\right)\) setwise.
	%, that is, 
	%\(g\left(\partial\mathrm{Fix}\left(z\right)\right)=\partial\mathrm{Fix}\left(z\right)\).
	Moreover, since $g$ is order-preserving, it must fix $t_i$ for each $i=1,\ldots,n$.
	As a consequence, we can restrict to study centralizers in each of
	the subgroups 
	\[
	H\left(\left[t_{i},t_{i+1}\right]\right)=\{h \in H \mid h(t)=t, \forall t \not \in
	[t_{i},t_{i+1}]\} \cong H,
	\]
	where \(i=0,1,\ldots,n-1\). If \(z\left(t\right)=t\) on 
	\(\left[t_{i},t_{i+1}\right]\), then it is easy to see that
	\(C_{H([t_i,t_{i+1}])}\left(z\right)=H([t_i,t_{i+1}])\cong H\). Otherwise,
	Corollaries \ref{thm:centralizers-affine} and \ref{cor:centralizers-translations}
	and 
	Propositions
	\ref{monod-centralizer-breakpoints-affine} and 
	\ref{thm:Centralizer-Last-Case}, cover the remaining cases (when $z$ is conjugate to an affine map or entirely above or below the diagonal) showing that either
	\(C_{H([t_i,t_{i+1}])}\left(z\right)\cong \left(\mathbb{R},+\right)\) or
	\(C_{H([t_i,t_{i+1}])}\left(z\right)\cong \left(\mathbb{Z},+\right)\).
\end{proof}
%REFERENCES
\bibliographystyle{acm}
\bibliography{references}
\end{document}